\documentclass[11pt]{article}
\usepackage{anysize}
\marginsize{2.0cm}{2.0cm}{2.0 cm}{2.0 cm}
\usepackage{setspace}
\usepackage{graphicx}
\usepackage{color}
\usepackage{epstopdf}
\usepackage[titletoc]{appendix}
\usepackage{graphicx}
\usepackage{latexsym,amsfonts,amsmath,amssymb}
\newtheorem{remark}{Remark}

\newtheorem{condition}{Condition}
\newtheorem{theorem}{Theorem}
\newtheorem{lemma}{Lemma}
\newtheorem{corollary}{Corollary}
\newtheorem{definition}{Definition}
\newtheorem{proposition}{Proposition}

\def\RR{\mathbb{R}}

\def\PP{\mathsf{P}}

\def\II{\mathbb{I}}

\def\EE{\mathsf{E}}
\def\PP{\mathsf{P}}

\begin{document}
	
	\title{Constrained optimal impulse control and inventory model}
	
	
	\author{Alexey Piunovskiy\thanks{Department of Mathematical Sciences, University of
			Liverpool, Liverpool, U.K. E-mail: piunov@liv.ac.uk.} }

	\maketitle
	
	\par\noindent\textbf{Abstract}. 
In this article, we consider the deterministic impulsively controlled system with infinite horizon and several discounted objective functionals. The constructed optimal control problem with functional constraints is reformulated as a Markov decision process, leading to (primal) convex and linear programs in the space of so-called occupation measures. We construct the dual programs and investigate the solvability of  all the  programs. Example of an inventory model illustrates the developed theory.

	\bigskip
	

	\section{Introduction}\label{sec1}
The literature on optimal impulse control is quite vast in terms of both theoretical developments and real-life applications. In the last years, along with the dynamic programming \cite{benk,jelito,PiunovskiySasha:2018,ppt,zhou}, the linear programming approach also appeared in \cite{helmes,b13pp,pz1,pz2}, and some duality issues were presented in \cite{helmes,pz2}. The Pontryagin maximum principle and similar methods were  developed in \cite{aru,leander,Miller:2003}. For applications of the optimal impulse control to, e.g., epidemiology and inventory, see \cite{benk,helmes,PiunovskiySasha:2018,ppt}.
There was no intention to present a survey on impulse control, we only tried to show the diversity of models and methods in the most relevant papers.

In the case of finite-dimensional Euclidean state space, sometimes
`impulsive control' means that
\begin{equation}\label{e171}
	x(t)=x_0+\int_0^t f(x(u))du+\int_0^t G(x(u))dw(u),
\end{equation}
where $x_0\in\RR^n$, $f(\cdot):~\RR^n\to\RR^n$ and the $(n\times m)$ matrix-valued function $G(\cdot)$ on  $\RR^n$ are fixed, and the function $w(\cdot):[0,\infty)\to\RR^m$ of bounded variation is to be chosen (i.e., plays the role of control).   A little more generally, one can  consider the models where $w(\cdot)$ is a (Borel) measure on $[0,\infty)$, and actually `impulses' correspond to the singular (with respect to the Lebesgue measure) component of $w(\cdot)$. Different modifications of the model (\ref{e171}) can be found in \cite{biba1,aru,b13pp}. It looks like such complicated models are not necessary for practical application, and many authors considered the following special case: 
$$	w(u)=\sum_{i=1}^\infty \sum_{j=1}^{i-1} b_j\II\{t_{i-1}\le u < t_i\}=\sum_{i=1}^\infty b_i\II\{t_i\le u\}$$
with $t_0:=0<t_1<t_2<\ldots$ and $b_j\ne 0$.
In other words, the (impulsive) deterministic control strategy is just the sequence $(\theta_1=t_1-t_0,b_1),(\theta_2=t_2-t_1,b_2),\ldots$. Such models appeared in \cite{benk,jelito,leander,Miller:2003,PiunovskiySasha:2018,ppt,pz1,pz2,zhou}; they are most similar  to the model investigated in the current article.

Roughly speaking, the linear programming approach  \cite{helmes,b13pp,pz1,pz2} means the following. The optimal control problem is replaced by the linear program in the space of so called occupation measures, which satisfy the specific characteristic equation, and the objective has the form of the integral of the cost function with respect to an occupation measure. After that, the optimal control strategy  can be retrieved from the optimal occupation measure. Note that dynamic programming can be also useful when investigating such linear programs \cite{pz2}.
	
Comparing with the existing studies, one can say the following.
\begin{itemize}
\item The underlying dynamical system was described as a flow in the metric space in \cite{jelito,PiunovskiySasha:2018,pz1,pz2}, but none of these works considered the constrained problem apart from \cite{pz1}, where the objectives were the total {\sl undiscounted} costs. 	It is well known \cite{PiunovskiySasha:2018} that discounted objective functional is a special case of the (undiscounted) total cost, but  for this special case, more general statements can be proved.
\item The functional constraints appeared in \cite{Miller:2003,pz1}, but the case of total {\sl discounted} costs was not considered there. By the way, the time horizon was finite in \cite{Miller:2003}. A rather general discounted model was studied in \cite{benk}, but without functional constraints.
\item Involving Markov decision processes is a rather new trick, used only in \cite{PiunovskiySasha:2018,pz1,pz2}, but again not for the case of discounted cost with functional constraints.
\item The duality issues were studied in \cite{helmes,pz2}, but only for the unconstrained case, and the objective functional was undiscounted in \cite{pz2}.
\item Finally, article \cite{jelito} is about the stochastic system with the specific objective functional with generalized discounting, but again without functional constraints.
The authors of \cite{helmes} investigated the very special model, namely the general stochastic inventory system, without functional constraints.
\end{itemize}
To summarize, the impulsively controlled deterministic dynamical system in a general Borel space with infinite horizon, total discounted costs and functional constraints (and even  without constraints) was not studied, and the current paper aims to fill these gaps. Note that the unconstrained model is a special case, and all the further presented results are applicable.
	
The rest of this article is organized as follows. Problem statement, reformulation as an MDP and preliminary results are given in Sections \ref{sec2} and \ref{sec3}. In the main sections \ref{sec5} and \ref{sec6}, two pairs of associated convex programs are investigated. In Section \ref{sec7} the developed theory is illustrated by the standard inventory model which certainly is not as general as in \cite{benk,helmes}. But the obtained constraint-optimal strategy exhibits the interesting property: one has to wait for a while before ordering the product, keeping the inventory at zero level: see Figure \ref{f1}.

The proof of all statements is postponed to Appendix.

Measures are allowed to be infinite-valued, $\RR_+:=[0,\infty)$, $Y^c$  is the complement of the set $Y$, $\delta_a(dy)$ is the Dirac measure concentrated at point $a$, w.r.t. means `with respect to', $e^{-\infty}:=0$.
	
	\section{Preliminary model description}\label{sec2}
	The {\sl state space} ${\bf X}$  \index{state space} of the system under investigation is a non-empty Borel subset of a complete separable metric space (i.e., Polish space)  with metric $\rho_X$ and the Borel $\sigma$-algebra ${\cal B}({\bf X})$. Typical examples are the finite-dimensional Euclidean spaces $\RR^n$. The dynamics is described by the {\sl flow} \index{flow} $\phi(\cdot,\cdot):~{\bf X}\times[0,\infty)\to{\bf X}$, where the second argument $t\in[0,\infty)$ is just time elapsed from the moment when the state was $x\in{\bf X}$, the first argument.  This flow must satisfy the following intuitively clear properties:
	\begin{itemize}
		\item $\phi(x,0)=x$ for all $x\in{\bf X}$;
		\item the mapping $\phi(\cdot,\cdot)$ is measurable;
		\item $\lim_{t\downarrow 0}\phi(\phi(x,s),t)=\phi(x,s)$ for all $x\in{\bf X}$ and $s\in[0,\infty)$;
		\item {\sl semigroup property}: \index{semigroup property} $\phi(x,(s+t))=\phi(\phi(x,s),t)$ for all $x\in{\bf X}$ and $(s,t)\in[0,\infty)\times[0,\infty)$.
	\end{itemize}
	The third property means that the trajectories of the system are continuous from the right. The fourth property means that the movement from $x\in{\bf X}$ to $\phi(x,s)$ on the time interval $(0,s]$ and the further movement from $\phi(x,s)$ to $\phi(\phi(x,s),t)$ on the time interval $(s,s+t]$ can be equivalently represented as the combined movement from $x$ to $\phi(x,s+t)$ on the interval $(0,s+t]$. For example, in case of ${\bf X}=\RR^n$ the flow can come from a well posed ordinary differential equation.
	
	Suppose at the zero time moment the state of the system was $x_0\in{\bf X}$. As time goes on, the process moves along the flow $\phi$ leading to the {\sl trajectory} $\{\phi(x_0,t),~t\ge 0\}$. We are going to study $J+1\ge 1$ objective criteria to be minimized, coming from the following components. There are measurable {\sl gradual} cost rates $C^g_j(x)\in[0,\infty)$ associated with any state $x\in{\bf X}$, $j=0,1,\ldots,J$. 
	At particular time moments $0\le t_1\le t_2\le\ldots<\infty$ the decision maker applies impulsive control actions $a_1,a_2,\ldots$, elements of the {\sl action space}  $\bf A$ being a non-empty Borel subset of a complete separable metric space with metric $\rho_A$ and the Borel $\sigma$-algebra ${\cal B}({\bf A})$. Each {\sl action} (or {\sl impulse})  $a\in{\bf A}$ applied at the current state $x\in{\bf X}$ results in the immediate jump of the process to the new state $l(x,a)\in{\bf X}$. The case when $l(x,a)=x$ for some $a\in{\bf A}$ is not excluded.
	The mapping $l(\cdot,\cdot):~{\bf X}\times{\bf A}\to{\bf X}$ is assumed to be measurable. Other consequences are the measurable {\sl lump sum} costs  $C^I_j(x,a)\in[0,\infty)$, $j=0,1,\ldots,J$.
	Below, $\theta_i:=t_i-t_{i-1},~i=1,2,\ldots;~t_0:=0$, so that $t_i=\sum_{k=1}^i\theta_k,~i=0,1,2,\ldots$. If there are $N<\infty$ impulses being applied at the time moments $t_1,t_2,\ldots,t_N$, then the $j$-th total associated discounted cost equals
	\begin{eqnarray}
		{\cal W}_j(x_0)&=&\sum_{i=1}^N\left\{\int_0^{\theta_i} e^{-\alpha(t_{i-1}+t)} C^g_j(\phi(x_{i-1},t))~dt+e^{-\alpha t_i}C^I_j(\phi(x_{i-1},\theta_i),a_i)\right\}\nonumber\\
		&&+\int_0^\infty e^{-\alpha(t_{N}+t)} C^g(\phi(x_N,t))~dt,~~~j=0,1,\ldots,J,\label{eq1}
	\end{eqnarray}
	where $\alpha>0$ is the fixed discount factor.
	In the above formula, $x_i:=l(\phi(x_{i-1},\theta_i),a_i)$, $i=1,2,\ldots,N$. 
	In general, the number of impulses is not bounded {\it a priori}, as it depends on the control strategy. In case $t_\infty:=\lim_{i\to\infty}t_i<\infty$, we don't consider the process after the time moment $t_\infty$. In what follows, we allow $t_i$ and $a_i$ to be randomized, so that in the rigorous formulae for the objectives, presented in the next section, mathematical expectation will appear.
	
	\section{Reformulation as an MDP and auxiliary results}\label{sec3}
	Firstly, we provide some informal preliminary observations.
	We add the artificial isolated point $\Delta$ to $\bf X$ and say that the control process is over at time moment $t_i$, without any future costs, if $x_i=\Delta$. Quite formally, $x_{i+1}=x_{i+2}=\ldots = \Delta$. Now the coefficient $e^{-\alpha t_{i-1}}$ in formula (\ref{eq1}) can be understood as the probability of the event $\Delta\ne x_{i-1}\in{\bf X}$.
	
	As soon as the sequence $\{(\theta_i,a_i)\}_{i=1}^\infty$ with $\theta_i\in[0,\infty)$, $a_i\in{\bf A}$ is fixed, the dynamics can be encoded as the following sequence
	$${\bf X}\ni x_0=X_0\to (\theta_1,a_1)\to X_1\to(\theta_2,a_2)\to\ldots.$$
	Here and below, capital letters are for random elements. If $X_{i-1}\in{\bf X}$ for $i=1,2,\ldots$, then $\PP(X_i=\Delta|X_{i-1}\in{\bf X})=1-e^{-\alpha \theta_i}$ and $\PP(X_i=l(\phi(X_{i-1},\theta_i),a_i)|X_{i-1}\in{\bf X})=e^{-\alpha\theta_i}$ leading to $\PP(X_i\in{\bf X})=e^{-\alpha t_i}$. The state $\Delta$ is costless and absorbing: $\PP(X_i=\Delta|X_{i-1}=\Delta)=1$. 
	If $\theta_{N+1}=\infty$ and $\theta_1,\theta_2,\ldots,\theta_N<\infty$, that is, only $N<\infty$ impulses are applied, then we  put $X_{N+1}=\Delta$, and the dynamics is encoded as
	$${\bf X}\ni x_0=X_0\to (\theta_1,a_1)\to X_1\to\ldots\to(\infty,a_{N+1})\to X_{N+1}=\Delta\to \ldots.$$
	Now formula (\ref{eq1}), corresponding to $N<\infty$ impulses, takes the form
	$${\cal W}_j(x_0)=\EE\left[\sum_{i=1}^{N+1}\bar C_j(X_{i-1},(\theta_i,a_i))\right]=\EE\left[\sum_{i=1}^{\infty}\bar C_j(X_{i-1},(\theta_i,a_i))\right],~~~j=0,1,\ldots,J,$$
	where $\theta_{N+1}=\infty$; $X_{N+1}=X_{N+2}=\ldots=\Delta$; $e^{-\alpha\infty}:=0$;
	$$\bar C_j(x,(\theta,a)):=\left\{\begin{array}{ll}
		\int_0^\theta e^{-\alpha t}C^g_j(\phi(x,t))dt
		+e^{-\alpha\theta}C^I_j(\phi(x,\theta),a), & \mbox{ if } x\in{\bf X};\\
		0, & \mbox{ if } x=\Delta.\end{array}\right.$$
	Since $C^g_j(\cdot),C^I_j(\cdot,\cdot)\ge 0$, all the integrals and mathematical expectations are well defined.
	
	After these preliminary observations, it becomes clear that the impulse control under study is just a Markov decision process (MDP) with the following elements:
	\begin{itemize}
		\item ${\bf X}_\Delta:={\bf X}\cup\{\Delta\}$ is the state space;
		\item ${\bf B}:=[0,\infty]\times{\bf A}$ is the action space equipped with the natural product $\sigma$-algebra;
		\item $Q(dy|x,b=(\theta,a)):=\left\{\begin{array}{ll}
			e^{-\alpha\theta}\delta_{l(\phi(x,\theta),a)}(dy)+(1-e^{-\alpha\theta})\delta_\Delta (dy), & \mbox{if } x\ne\Delta,\theta\ne\infty;\\
			\delta_\Delta(dy) & \mbox{otherwise} \end{array}\right.$ is the transition probability;
		\item $\bar C_j(x,b=(\theta,a))$, $j=0,1,\ldots,J$ are the one-step costs.
	\end{itemize}
	
	As usual, we introduce trajectories
	$$\omega=(x_0,b_1=(\theta_1,a_1),x_1,b_2=(\theta_2,a_2),x_2,\ldots)\in({\bf X}_\Delta\times{\bf B})^\infty=:\Omega;$$
	$\Omega$ is the sample space. Capital letters $X_i,B_i,\Theta_i,T_i=\sum_{k=1}^i \Theta_i,A_i$ denote the corresponding functions (projections) of $\omega\in\Omega$, i.e., random elements. $\cal F$ is the natural product $\sigma$-algebra on $\Omega$. Finite sequences of the form
	$$h_i=(x_0,b_1=(\theta_1,a_1),x_1,b_2=(\theta_2,a_2),\ldots,x_i)$$
	will be called $i$-histories, $i=0,1,\ldots$. The space of all such $i$-histories will be denoted as ${\bf H}_i$. We endow it with the $\sigma$-algebra ${\cal B}({\bf H}_i)$, which is the restriction of $\cal F$ to ${\bf H}_i$.
	
	\begin{definition}\label{d1}
		\begin{itemize}
			\item[(a)]
			A (control) strategy $\pi=\{\pi_i\}_{i=1}^\infty$ is a sequence of stochastic kernels $\pi_i$ on $\bf B$ given $\textbf{H}_{i-1}$.
			\item[(b)] A strategy $\pi$ is called stationary and denoted as $\pi^s$ if $\pi_i(db|h_{i-1})=\pi^s(db|x_{i-1})$.
			\item[(c)] A strategy ${\pi}$ is called deterministic stationary and denoted as $f$, if for all $i=1,2,\ldots$, $\pi_i(db|h_{i-1})=\delta_{f(x_{i-1})}(db)$, where $f:\textbf{X}_\Delta\to{\bf B}$ is a measurable mapping.
		\end{itemize}
	\end{definition}
	
	In terms of interpretation, under a strategy $\pi=\{\pi_i\}_{i=1}^\infty$, $\pi_i(db|h_{i-1})$ is the (regular) conditional distribution of $B_i=(\Theta_{i},A_i)$ given the $(i-1)$-history $H_{i-1}=h_{i-1}$. This is in line with the following construction. 
	For a given initial distribution $\nu$ on $\textbf{X}_\Delta$ and a strategy $\pi$, by the Ionescu-Tulcea Theorem, see e.g., \cite[Prop.7.28]{Bertsekas:1978}, there is a unique probability measure $\PP^\pi_{\nu}$ on $(\Omega,{\cal F})$ called {\sl strategic} and satisfying the following conditions:
	\begin{eqnarray*}
		\PP^\pi_{\nu}(X_0\in\Gamma_X)=\nu(\Gamma_X)~\forall~ \Gamma_X\in{\cal B}({\bf X}_\Delta)
	\end{eqnarray*}
	and for all $i=1,2,\dots$, $\Gamma_B\in{\cal B}({\bf B})$, $\Gamma_X\in{\cal B}({\bf X}_\Delta)$, $\PP^\pi_{\nu}$-almost surely
	\begin{eqnarray}
		\PP^\pi_{\nu}(B_i\in\Gamma_B|H_{i-1})&=&\pi_i(\Gamma_B|H_{i-1});\label{e33}\\
		\PP^\pi_{\nu}(X_i\in\Gamma_X|H_{i-1},B_i=(\Theta_i,A_i))
		&=& Q(\Gamma_X| X_{i-1},B_i)  \nonumber
	\end{eqnarray}
	$$=\left\{\begin{array}{ll} e^{-\alpha\Theta_i}
		\delta_{l(\phi(X_{i-1},\Theta_i),A_i)}(\Gamma_X)\\
		+(1-e^{-\alpha\Theta_i})\delta_\Delta(\Gamma_X), & \mbox{ if } X_{i-1}\ne\Delta,~\Theta_i\ne\infty; \\
		\delta_\Delta(\Gamma_X) & \mbox{ otherwise. } \end{array} \right.$$
	When the initial distribution $\nu$ is a Dirac measure concentrated on a singleton, say $\{x_0\}$, we write $\PP^\pi_\nu$ as $\PP^\pi_{x_0}$. The mathematical expectation with respect to $\PP^\pi_\nu$ and $\PP^\pi_{x_0}$ is denoted as $\EE^\pi_\nu$ and $\EE^\pi_{x_0}$, respectively. Note that $\PP^\pi_\nu(X_i\in{\bf X})=\EE^\pi_\nu[e^{-\alpha T_i}]$, $i=0,1,\ldots$.
	
	Now the objectives (performance functionals) are
	$${\cal V}_j(\nu,\pi):=\EE^\pi_\nu\left[\sum_{i=1}^\infty\bar C_j(X_{i-1},(\Theta_i,A_i))\right], ~j=0,1,\ldots,J.$$
	Again, when $\nu=\delta_{x_0}$, we write ${\cal V}_j(x_0, \pi)$ for ${\cal V}_j(\nu, \pi)$. Note that we do not exclude the possibility of $\sum_{i=1}^{\infty}\Theta_i<\infty$, but, as already  mentioned above, we will only consider the total cost accumulated over $[0,\sum_{i=1}^\infty \Theta_i).$ This is consistent with the definition of ${\cal V}_j(\nu, \pi)$.
	
	The constrained optimal control problem under study is the following one:
	\begin{eqnarray}\label{PZZeqn02}
		\mbox{Minimize with respect to } \pi &&{\cal V}_0(x_0, \pi)  \\
		\mbox{subject to }&& {\cal V}_j(x_0,\pi)\le d_j,~j=1,2,\dots,J.\nonumber
	\end{eqnarray}
	Here and below, we take $x_0\in\textbf{X}$ as a fixed initial point, and $d_j\ge 0$, $j=1,2,\ldots, J$ as fixed constraint constants.

	\begin{definition}\label{d2}
		A strategy $\pi$ is called {\sl feasible} if it satisfies all the constraint inequalities in problem (\ref{PZZeqn02}). A feasible strategy $\pi^\ast$ is called {\sl optimal} if ${\cal V}_0(x_0,\pi^*)\le  {\cal V}_0(x_0,\pi)$ for all feasible strategies $\pi$. The infimum of ${\cal V}_0(x_0,\pi)$ over all feasible strategies is called the {\sl  optimal value} of problem (\ref{PZZeqn02}), denoted as val(\ref{PZZeqn02}).
	\end{definition}
	
	\begin{condition}\label{con1}
		The cost functions $C^g_j(\cdot)$ and $C^I_j(\cdot,\cdot)$, $j=0,1,\ldots,J$ are non-negative.
	\end{condition}
	
	\begin{condition}\label{ConstrainedPPZcondition05}
		There exists some feasible strategy $\pi$ such that ${\cal V}_0(x_0,\pi)<\infty.$ We say that $\pi$ is a feasible strategy with finite value.
	\end{condition}
	
	If Condition \ref{ConstrainedPPZcondition05} is not satisfied, then either problem (\ref{PZZeqn02}) is not solvable or any feasible strategy is optimal. 
	Conditions \ref{con1} and  \ref{ConstrainedPPZcondition05} imply that $d_j\ge 0$ for all $j=1,2,\ldots,J$.
	
	\begin{condition}\label{ConstrainedPPZcondition01}
		\begin{itemize}
			\item[(a)] The space $\textbf{A}$ is compact, and $+\infty$ is the one-point compactification of $[0,\infty)$.
			\item[(b)] The function $(x,a)\in \textbf{X}\times \textbf{A}\to l(x,a)$ is continuous.
			\item[(c)] The function $(x,\theta)\in \textbf{X}\times [0,\infty)\to \phi(x,\theta)$ is continuous.
			\item[(d)] For each $j=0,1,\dots,J,$ the function $(x,a)\in \textbf{X}\times \textbf{A}\to C_j^I(x, a)$ is lower semicontinuous.
			\item[(e)] For each $j=0,1,\dots,J,$ the function $x\in \textbf{X}\to C_j^g(x)$ is lower semicontinuous.
		\end{itemize}
	\end{condition}
	
	\begin{proposition}\label{pr1}
		Under Conditions \ref{con1} and \ref{ConstrainedPPZcondition01}, the MDP is {\sl positive semicontinous} in the sense that the action space $\bf B$ is compact, $Q$ is a continuous stochastic kernel, and the cost functions $\bar C_j(\cdot,\cdot)$ are non-negative lower semicontinuous for all $j=0,1,\ldots, J$.
	\end{proposition}
	
	\begin{lemma}\label{l1}
		If in a positive semicontinuous MDP there is a feasible strategy $\pi$ with a finite value, then there exists a stationary strategy which is feasible and optimal in problem  (\ref{PZZeqn02}).
	\end{lemma}
	
	For the proof, see Theorem 4.1 of \cite{Dufour:2012}.
	
	In what follows, we assume that the MDP under investigation is positive semicontinuous.
	
	Let
	\begin{eqnarray}\label{e5p}
		{\bf V}:=\left\{x\in\textbf{X}_\Delta: \inf_{\pi}\EE_x^\pi\left[ \sum_{i=1}^\infty \sum_{j=0}^J\bar{C}_j(X_{i-1},B_{i}) \right]>0\right\},
	\end{eqnarray}
	Note that $\Delta\in {\bf V}^c:=\textbf{X}_\Delta\setminus {\bf V}$ because $\Delta$ is a costless cemetery in this MDP.
	
	When applying Corollary 9.17.2 of \cite{Bertsekas:1978} to the MDP under study with the cost function $\sum_{j=0}^J \bar C_j(x,b)$, we obtain that the Bellman function
	\begin{equation}\label{e8}
		V(x):=\inf_\pi \EE^\pi_x\left[\sum_{i=1}^\infty \sum_{j=0}^J \bar C_j(X_{i-1},B_i)\right],~~~x\in{\bf X}_\Delta
	\end{equation}
	is a non-negative lower semicontinuous function if the MDP model  is positive semicontinuous. In particular, the set ${\bf V}^c$ is a closed subset of ${\bf X}_\Delta$ and $\bf V$ is an open subset of $\bf X$. Moreover, there is a deterministic stationary strategy $f^*$ providing the infimum in (\ref{e8}) and in the Bellman equation
	$$V(x)=\inf_{b\in{\bf B}}\left\{\sum_{j=0}^J \bar C_j(x,b)+\int_{{\bf X}_\Delta}V(y)Q(dy|x,b)\right\},~~~~~x\in{\bf X}_\Delta,$$
	and 
	\begin{equation}\label{e5}
		\bar C_j(x,f^*(x))=Q({\bf V}|x,f^*(x))=0~~~~~\forall x\in{\bf V}^c,~~\forall j=0,1,\ldots, J.
	\end{equation}
	For the last statement, it is sufficient to notice that for all $x\in{\bf V}^c$
	$$0=V(x)=\sum_{j=0}^J\bar C_j(x,f^*(x))+\int_{\bf V} V(y)Q(dy|x,f^*(x))$$
	and $V(y)>0$ for $y\in{\bf V}$.
	
	Now it is clear that, if $x_0\in{\bf V}^c$, then $f^*$ is the solution to problem  (\ref{PZZeqn02}):
	$${\cal V}_j(x_0,f^*)=0~~~~~\forall x_0\in{\bf V}^c,~~j=0,1,\ldots, J.$$
	It is also clear that, at all decision epochs $i=1,2,\ldots$, it is optimal to apply actions $f^*(X_{i-1})$ as soon as $X_{i-1}\in{\bf V}^c$: the expected future costs associated with $\bar C_j(\cdot,\cdot)$, $j=0,1,\ldots,J$, are zero.
	
	\begin{definition}\label{d3}
		Assuming that the MDP model is positive and semicontinuous and there exists a  feasible strategy with finite value,
		$\Pi$ is the class of feasible strategies with finite value and satisfying the condition 
		\begin{equation}\label{e0} 
			\pi_i(db|h_{i-1}):=\delta_{f^*(x_{i-1})}(db),~~~~~\forall i=1,2,\ldots,
		\end{equation}
		if $x_{i-1}$, the last component of $h_{i-1}$, is in ${\bf V}^c$.
	\end{definition}
	
	Now it is possible to restrict to the strategies from $\Pi$ and  concentrate on the subset $\bf V$. In particular, we assume that $x_0\in{\bf V}$. Note also that, without loss of generality, one can say that the feasible and optimal in problem (\ref{PZZeqn02}) strategy $\pi$, existing by Lemma \ref{l1}, is from the class $\Pi$.
	
	\begin{definition}\label{15JuneDef02}
		Let the initial state $x_0\in{\bf X}$  be fixed.
		For each strategy $\pi$, its {\sl occupation measure} $\mu^\pi$  is defined by
		$$\mu^\pi(\Gamma_1\times\Gamma_2):=\EE_{x_0}^\pi\left[\sum_{i=0}^\infty \II\{(X_i,B_{i+1})\in\Gamma_1\times\Gamma_2\}\right]~\forall~\Gamma_1\in{\cal B}(\textbf{X}_\Delta),\Gamma_2\in{\cal B}(\textbf{B}).$$
	\end{definition}
	
	\begin{lemma}\label{l5} For a fixed initial state $x_0\in{\bf X}$,
		the set of all occupation measures is convex.
	\end{lemma}
	For the proof see  Propositions 1 and 3 of \cite{sicon24}.
	
	By the monotone convergence theorem, for $\bar C_j(\cdot,\cdot)\ge 0$
	$${\cal V}_j(x_0,\pi)=\int_{{\bf X}_\Delta\times{\bf B}} \bar C_j(x,b)\mu^\pi(dx\times db), ~j=0,1,\ldots, J.$$
	As is well known (see Lemma 9.4.3 of \cite{HernandezLerma:1999}), each occupation measure satisfies equation
	\begin{eqnarray}\label{enum1}
		\mu^\pi(\Gamma\times \textbf{B})=\delta_{x_0}(\Gamma)+\int_{\textbf{X}_\Delta\times \textbf{B}}Q(\Gamma|y,b)\mu^\pi(dy\times db),~~~\Gamma\in{\cal B}(\textbf{X}_\Delta).
	\end{eqnarray}
	
	\begin{lemma}\label{l2}
		Suppose  the MDP model is positive and semicontinuous. Then each measure $\mu$ on ${\bf X}_\Delta\times{\bf B}$, satisfying equation (\ref{enum1}) and such that 
		$$\displaystyle \int_{{\bf X}_\Delta\times{\bf B}}\sum_{j=0}^J \bar C_j(x,b)\mu(dx\times db)<\infty,$$ 
		exhibits the following property:
		\begin{equation}\label{enum2}
			\mbox{the restriction on $\bf V$ of the marginal measure $\mu(dx\times{\bf B})$ is  $\sigma$-finite.}
		\end{equation}
	\end{lemma}
	
	For the proof, see Theorem 3.2 of \cite{Dufour:2012}.
	
	According to the explanations below Lemma \ref{l1}  (see (\ref{e5})), we always assume that condition (\ref{e0}) is satisfied. Now $\int_{{\bf V}^c\times{\bf B}}\bar C_j(x,b)\mu^\pi(dx\times db)\equiv 0$ for all $j=0,1,\ldots, J$, and one is interested only in the restriction of $\mu^\pi$ on ${\bf V}\times{\bf B}$. In what follows, with some abuse of notation, such restrictions of occupation measures (and of any measures $\mu$ on ${\bf X}_\Delta\times{\bf B}$) are also sometimes denoted as $\mu^\pi$ (and as $\mu$).
	
	\begin{corollary}\label{cor1}
		Under the conditions of Lemma \ref{l2}, if $\pi$ is a feasible strategy with finite value, then the occupation measure $\mu^\pi$ exhibits property (\ref{enum2}). In particular, for each strategy $\pi\in\Pi$, the measure $\mu^\pi$ exhibits property (\ref{enum2}).
	\end{corollary}
	
	The proof is obvious.
	
	If $\pi\in\Pi$, then for $\Gamma_1\subset {\bf V}^c$
	$$\mu^\pi(\Gamma_1\times\Gamma_2)=\EE_{x_0}^\pi\left[\sum_{i=0}^\infty \II\{(X_i \in\Gamma_1\}\II\{\Gamma_2\ni f^*(X_i)\}\right]~\forall~\Gamma_2\in{\cal B}(\textbf{B}).$$
	Therefore,
	$$\int_{{\bf V}^c\times{\bf B}} \bar C_j(x,b)\mu^\pi(dx\times db)=\EE^\pi_{x_0}\left[\sum_{i=0}^\infty \II\{X_i\in{\bf V}^c\} \bar C_j(X_i,f^*(X_i))\right]=0,~
	j=0,1,\ldots, J$$
	by (\ref{e5}), and 
	$${\cal V}_j(x_0,\pi)=\int_{{\bf V}\times{\bf B}} \bar C_j(x,b)\mu^\pi(dx\times db), ~j=0,1,\ldots, J.$$
	Furthermore, for $\pi\in\Pi$,  the measure $\mu^\pi$ on ${\bf V}\times{\bf B}$ satisfies equation
	$$\mu^\pi(\Gamma\times \textbf{B})=\delta_{x_0}(\Gamma)+\int_{\textbf{V}\times \textbf{B}}Q(\Gamma|y,b)\mu^\pi(dy\times db),~~\Gamma\in{\cal B}({\bf V})$$
	because for $\Gamma\subset {\bf V}$,  again by (\ref{e5}),
	$$\int_{\textbf{V}^c\times \textbf{B}}Q(\Gamma|y,b)\mu^\pi(dy\times db)=\EE^\pi_{x_0}\left[\sum_{i=0}^\infty \II\{X_i\in{\bf V}^c\} Q(\Gamma|X_i,f^*(X_i))\right]=0.$$
	
	\begin{definition}\label{d4}
		Suppose  the MDP model is positive and semicontinuous. We say that a stationary strategy $\pi^\mu$ is {\sl induced} by a  measure $\mu$ on ${\bf V}\times{\bf B}$ with property (\ref{enum2}),  if it has the following form
		$$\pi^\mu(db|x)=\II\{x\in{\bf V}\}\varphi_\mu(db|x)+\II\{x\in{\bf V}^c\}\delta_{f^*(x)}(db),$$
		where $\varphi_\mu$ is a stochastic kernel on $\bf B$ given $\bf V$ satisfying
		$$\varphi_\mu(db|x)\mu(dx\times{\bf B})=\mu(dx\times db)~~~\mbox{ on } {\cal B}({\bf V}\times{\bf B}).$$
	\end{definition} 
	The stochastic kernel $\varphi_\mu$ exists by Corollary 7.27.1 of \cite{Bertsekas:1978} which should be applied to the sets $\Gamma\times{\bf B}$ with $\mu(\Gamma\times{\bf B})<\infty$. There may be many induced strategies, different on  $\mu(dx\times{\bf B})$-zero subsets of $\bf V$.
	
	\begin{definition}\label{d7}
		$\cal D$ is the set of measures  $\mu$ on ${\bf V}\times{\bf B}$ with $\sigma$-finite marginals $\mu(dx\times {\bf B})$,
		satisfying equation 
		\begin{equation}\label{e9}
			\mu(\Gamma\times \textbf{B})=\delta_{x_0}(\Gamma)+\int_{{\bf V}\times \textbf{B}}Q(\Gamma|y,b)\mu(dy\times db),~\Gamma\in{\cal B}({\bf V})\
		\end{equation}
		and such that
		\begin{equation}\label{e13p}
			\int_{{\bf V}\times{\bf B}} \sum_{j=0}^J \bar C_j(x,b)\mu(dx\times db)<\infty.
		\end{equation}
	\end{definition}
	
	\begin{lemma}\label{l12}
		There exists a vector space $[U]$ such that $\cal D$ can be identified with a convex subset of $[U]$.
	\end{lemma}
	
	\section{First pair of convex programs}\label{sec5}
	
	We are ready to formulate the first pair of convex programs under Condition \ref{ConstrainedPPZcondition05} in case the MDP model is positive and semicontinuous. After introducing the {\sl Lagrangian}
	$$L_1(\mu,\bar g):=\int_{{\bf V} \times\textbf{B}}\bar{C}_0(x,b)\mu(dx\times db)+\sum_{j=1}^J g_j\left(\int_{{\bf V} \times\textbf{B}}\bar{C}_j(x,b)\mu(dx\times db)-d_j\right)$$
	for $\mu\in{\cal D}$, $\bar g:=(g_1,g_2,\ldots, g_J)\in\RR_+^J$, the primal and dual programs look, correspondingly, as follows:
	\begin{eqnarray}\label{enum4p}
		\mbox{Minimize over }~\mu\in{\cal D}&:&\sup_{\bar g\in\RR_+^J} L_1(\mu,\bar g);\\
		\label{enum4}
		\mbox{Maximize over }~ \bar g\in\RR_+^J&:&\inf_{\mu\in{\cal D}} L_1(\mu,\bar g).
	\end{eqnarray}
	\begin{itemize}
\item	val(\ref{enum4p})$:=\inf_{\mu\in{\cal D}}\sup_{\bar g\in\RR^J_+} L_1(\mu,\bar g)$. Similar notation is accepted for all other linear and convex programs.\\
\item	$h(\bar g):= \inf_{\mu\in{\cal D}} L_1(\mu,\bar g)$ is called the {\sl dual functional}.
\end{itemize}
	In the explicit form the primal program (\ref{enum4p}) can be rewritten as
	\begin{eqnarray}
		\mbox{Minimize}&:&\int_{{\bf V} \times\textbf{B}}\bar{C}_0(x,b)\mu(dx\times db) ~\mbox{ over } \mu\in{\cal D}\label{e13pp}\\
		\mbox{subject to}&:& \int_{{\bf V} \times\textbf{B}}\bar{C}_j(x,b)\mu(dx\times db) - d_j\le 0,~~~j=1,2,\dots,J.\label{e14pp}
	\end{eqnarray}
	
	\begin{remark}\label{rem71}
		If the MDP model is positive and semicontinuous and there is a feasible strategy $\pi$ with a finite value, then problem (\ref{PZZeqn02}) restricted to the strategies from $\Pi$ is equivalent to the program (\ref{e13pp})-(\ref{e14pp}) with ${\cal D}$ replaced by $\{\mu^\pi,~\pi\in\Pi\}\subset{\cal D}$: see Corollary \ref{cor1}.
	\end{remark}
	
	\begin{lemma}\label{l3}
		Suppose  the MDP model is positive and semicontinuous. Let a measure $\mu$ be feasible in the primal program  (\ref{e13pp})-(\ref{e14pp}). Then for each $j=0,1,\ldots, J$
		$$\EE^{\pi^\mu}_{x_0}\left[\sum_{i=1}^\infty \bar C_j(X_{i-1},B_i)\right]=\int_{{\bf V}\times{\bf B}} \bar C_j(x,b)\mu^{\pi^\mu}(dx\times db)\le\int_{{\bf V}\times{\bf B}} \bar C_j(x,b)\mu(dx\times db).$$
	Here $\pi^\mu$ is the induced strategy as in Definition \ref{d4}, and $\mu^{\pi^\mu}$ is its occupation measure  as in Definition \ref{15JuneDef02}. Note, $\mu$ and $\mu^{\pi^\mu}$ may be different.
	\end{lemma}
	
	For the proof, see Corollary 3.1 of \cite{Dufour:2012}. In that corollary, integration was over ${\bf X}_\Delta\times{\bf B}$, but, being restricted to the measures from $\cal D$, one can  accept that ${\bf V}^c$ is the  absorbing set with zero costs $\bar C_j(\cdot,\cdot)$. Thus integration over ${\bf V}^c$ can be ignored.  Note that the strategy $\pi^\mu$, as in Lemma \ref{l3}, belongs to the class $\Pi$.
	
	\begin{theorem}\label{t1}
		Suppose the model is positive and semicontinuous and there is a feasible strategy $\pi$ with a finite value. Then the following assertions hold.
		\begin{itemize}
			\item[(a)] val(\ref{PZZeqn02})$=$val(\ref{enum4p}) is finite.
			\item[(b)] If a strategy $\pi^*\in\Pi$ is feasible and optimal in problem (\ref{PZZeqn02}), then the measure $\mu^{\pi^*}$ is feasible and optimal in the primal  program (\ref{e13pp})-(\ref{e14pp}). Note that $\pi^*$ exists by Lemma~\ref{l1}.
			\item[(c)] If a measure $\mu^*$ is feasible and optimal in the primal  program  (\ref{e13pp})-(\ref{e14pp}), then the induced strategy $\pi^{\mu^*}$ is feasible and optimal in problem (\ref{PZZeqn02}) and $\pi^{\mu^*}\in\Pi$.
		\end{itemize}
	\end{theorem}
	
	\begin{condition}\label{con4} (Slater condition.) There exists $\mu\in{\cal D}$ such that
		$$ \int_{{\bf V} \times\textbf{B}}\bar{C}_j(x,b)\mu(dx\times db) < d_j,~~~j=1,2,\dots,J.$$
	\end{condition}
	
	\begin{theorem}\label{prop2}
		Suppose the model is positive and semicontinuous,  there is a feasible strategy $\pi$ with a finite value, and the Slater condition is satisfied. Then the following assertions hold.
		\begin{itemize}
			\item[(a)] There is no duality gap: $\inf_{\mu\in{\cal D}}\sup_{\bar g\in\RR_+^J} L_1(\mu,\bar g)=\sup_{\bar g\in\RR_+^J} \inf_{\mu\in{\cal D}} L_1(\mu,\bar g)$, and there exists at least one $\bar g^*\in\RR_+^J$ solving the dual convex program (\ref{enum4}). The dual functional $h(\cdot)$ is concave.
			\item[(b)] A point $\mu^*\in{\cal D}$ is an optimal (and feasible) solution to the primal convex program (\ref{enum4p}) and $\bar g^*\in\RR_+^J$ is an optimal solution to the dual convex program (\ref{enum4})
			if and only if one of the following two equivalent statements holds. 
			\begin{itemize}
				\item[(i)]  The pair $(\mu^*,\bar g^*)$ is a {\sl saddle point} of the Lagrangian $L_1$:
				\begin{eqnarray*}
					L_1(\mu^*,\bar g)\le L_1(\mu^*,\bar g^*)\le L_1(\mu,\bar g^*),~\forall~ \mu\in{\cal D},~\bar g\in\RR_+^J.
				\end{eqnarray*}
				\item[(ii)]
				The following relations hold:
				\begin{eqnarray*}
					&&\int_{{\bf V} \times\textbf{B}}\bar{C}_j(x,b)\mu^*(dx\times db) \le d_j,~j=1,2,\ldots, J;\\
					&& L_1(\mu^*,\bar g^*)=\inf_{\mu\in{\cal D}} L_1(\mu,\bar g^*)=\int_{{\bf V} \times\textbf{B}}\bar{C}_0(x,b)\mu^*(dx\times db);\\
					&& \sum_{j=1}^J g^*_j\left( \int_{{\bf V} \times\textbf{B}}\bar{C}_j(x,b)\mu^*(dx\times db)-d_j\right)=0.
				\end{eqnarray*}
				The last equality is known as the {\sl Complementary Slackness Condition}. If $\bar g^*$ is known to be optimal in the dual convex program (\ref{enum4}) and the first two conditions in (ii) are fulfilled, then the complementary slackness condition is satisfied automatically.
			\end{itemize}
		\end{itemize}
	\end{theorem}
	
	For the proof see  Section 8.6 of \cite{b42} and Theorem 2 of \cite{rock}.  Statements from \cite{b42} are applicable by Lemma \ref{l12}.
	
	\begin{remark}\label{rem3}
		Suppose all the conditions in Theorem \ref{prop2}  are satisfied. 
		\begin{itemize}
			\item[(a)] For a fixed $\bar g\in\RR_+^J$ the problem $L_1(\mu,\bar g)\to\inf_{\mu\in{\cal D}}$ (up to the constant $\sum_{j=1}^J g_jd_j$) is equivalent to the program (\ref{enum4p}) with $J=0$, where $\bar C_0(\cdot,\cdot)$ is replaced by $\bar C_0(\cdot,\cdot)+\sum_{j=1}^J g_j\bar C_j(\cdot,\cdot)$. Therefore, by Theorem \ref{t1}(a)
			\begin{eqnarray*}
				h(\bar g)&:=&\inf_{\mu\in{\cal D}} L_1(\mu,\bar g)\\
				&=&\inf_{\pi} \EE^\pi_{x_0}\left[\sum_{i=1}^\infty\left( \bar C_0(X_{i-1},B_i)+\sum_{j=1}^Jg_j\bar C_j(X_{i-1},B_i)\right)\right]-\sum_{j=1}^J g_jd_j.
			\end{eqnarray*}
			\item[(b)] Let $\mu^*$ be the optimal measure in the primal convex program (\ref{enum4p}), which exists by Lemma \ref{l1} and Theorem \ref{t1}(b). For $\bar g^*\in\RR_+^J$ solving the  dual program (\ref{enum4}), the common optimal value of the both programs (\ref{enum4p}) and (\ref{enum4}) equals
			\begin{eqnarray*}
				&&\int_{{\bf V}\times{\bf B}}\bar C_0(x,b)\mu^*(dx\times db)
				=\inf_{\mu\in{\cal D}} L_1(\mu,\bar g^*)\\
				&=&\inf_\pi\EE^\pi_{x_0}\left[\sum_{i=1}^\infty\left( \bar C_0(X_{i-1},B_i)+\sum_{j=1}^Jg^*_j\bar C_j(X_{i-1},B_i)\right)\right]-\sum_{j=1}^J g_j^*d_j:
			\end{eqnarray*}
			see Theorem \ref{prop2}(b-ii).
		\end{itemize}
	\end{remark}
	
	\section{Second pair of convex programs}\label{sec6}
	
	\begin{condition}\label{con5}\begin{itemize}
			\item[(a)] $C^I_0(x,a)\ge \delta>0$ for all $x\in{\bf X},~a\in{\bf A}$. 
			\item[(b)] ${\cal C}:=\max_{j=0,1,\ldots,J}\sup_{x\in{\bf X}} C^g_j(x)+\max_{j=0,1,\ldots,J}\sup_{(x,a)\in{\bf X}\times{\bf A}} C^I_j(x,a)<\infty$.
		\end{itemize}
	\end{condition}
	
	Under Condition \ref{con5}(a), if the MDP under investigation is positive and semicontinuous, then,
	for each $x\in{\bf V}^c\setminus\{\Delta\}$, necessarily by (\ref{e5}) $f^*(x)=(\infty,\hat a)$ with $\hat a\in{\bf A}$ being arbitrary. Hence, in this case, for each $x\in{\bf V}^c\setminus\{\Delta\}$ $C^g_j(\phi(x,t))\equiv 0$ almost everywhere with respect to the Lebesgue measure for all $j=0,1,\ldots, J$. One can also say that $x\in{\bf V}^c$ if and only if $C^g_j(\phi(x,t))\equiv 0$ almost everywhere with respect to the Lebesgue measure for all $j=0,1,\ldots,J$, and $Q(\{\Delta\}|x,f^*(x))=1$.
	
	\begin{lemma}\label{l4}
		Suppose the model is positive and Condition \ref{con5}(a) is satisfied. Then for each given strategy $\pi$, if $\mu^\pi({\bf X}\times{\bf B})=\infty$, then ${\cal V}_0(x_0,\pi)=\infty$. 
	\end{lemma}
	
	According to Lemma \ref{l4}, under Conditions   \ref{ConstrainedPPZcondition05}  and \ref{con5}(a), if the MDP model is positive and semicontinuous, then one can investigate the primal  program (\ref{enum4p}) in the space of finite measures on ${\bf V}\times{\bf B}$. 
	
	\begin{definition}\label{d5}
		Assuming that Conditions \ref{ConstrainedPPZcondition05} and \ref{con5}(a) are satisfied and the MDP model is positive and semicontinuous, we denote
		${\cal M}^f$ the space of all the restrictions of  occupation measures on ${\bf V}\times{\bf B}$ which are finite.
	\end{definition}
	
	Now, under the conditions  in Definition \ref{d5}, the primal  program (\ref{e13pp})-(\ref{e14pp}) can be rewritten in the following way:
	\begin{eqnarray}\label{e21}
		\mbox{Minimize}&:&\int_{{\bf V} \times\textbf{B}}\bar{C}_0(x,b)\mu(dx\times db)
		\mbox{ over finite measures $\mu$ on ${\bf V}\times \textbf{B}$} \\
		\mbox{subject to}&:& \mu(\Gamma\times \textbf{B})=\delta_{x_0}(\Gamma)+\int_{{\bf V}\times \textbf{B}}Q(\Gamma|y,b)\mu(dy\times db),~\Gamma\in{\cal B}({\bf V})\label{e22}\\
		\mbox{and}&&\int_{{\bf V} \times\textbf{B}}\bar{C}_j(x,b)\mu(dx\times db)\le d_j,~j=1,2,\dots,J. \label{e23}
	\end{eqnarray}
	
	\begin{definition}\label{d8}
		${\cal D}^f$ is the space of finite measures $\mu$ on ${\bf V}\times{\bf B}$ and $\cal W$ is the (linear) space of bounded measurable functions $W(\cdot)$ on $\bf V$. Clearly, ${\cal D}^f$ is the positive cone in the linear space of finite signed measures on ${\bf V}\times{\bf B}$.
		The Lagrangian on ${\cal D}^f\times{\cal W}\times\RR^j_+$ is defined by
		\begin{eqnarray*}
			L_2(\mu,W,\bar g)&:=& \int_{{\bf V}\times{\bf B}} \bar C_0(x,b)\mu(dx\times db)\\
			&&+\int_{\bf V} W(x)\left[\delta_{x_0}(dx)+\int_{\bf V} Q(dx|y,b)\mu(dy\times db)-\mu(dx\times{\bf B})\right]\\
			&&+\sum_{j=1}^J g_j\left(\int_{{\bf V}\times{\bf B}}\bar C_j(x,b)\mu(dx\times db)-d_j\right).
		\end{eqnarray*}
	\end{definition}
	
	Since $\bar C_j(\cdot,\cdot)$ and $g_j$ are non-negative, the Lagrangian $L_2$ is well defined and the terms therein can be arbitrarily rearranged.
	
	It is straightforward to check that the primal convex (in fact, linear) program (\ref{e21})-(\ref{e23}) is equivalent to
	\begin{equation}\label{e24}
		\mbox{Minimize over }~\mu\in{\cal D}^f~:~~\sup_{(W(\cdot),\bar g)\in{\cal W}\times\RR_+^J} L_2(\mu,W,\bar g).
	\end{equation}
	
	\begin{remark}\label{r2}
		If equality (\ref{e22}) is violated, then there exist $\varepsilon>0$ and $\Gamma\in{\cal B}({\bf V})$ such that
		$$\delta_{x_0}(\Gamma)+\int_{\bf  V} Q(\Gamma|y,b)\mu(dy\times db)-\mu(\Gamma\times{\bf B})\ge\varepsilon~(\mbox{or } \le-\varepsilon).$$
		Now for $W_N(x)=N\II\{x\in\Gamma\}$ (or for $W_N(x)=-N\II\{x\in\Gamma\}$) we see that \linebreak
		$\lim_{N\to\infty} L_2(\mu,W,\bar g)=+\infty$. The similar reasoning applies to constraints (\ref{e23}). At the same time, according to Condition  \ref{ConstrainedPPZcondition05}, there exists $\mu^\pi\in{\cal D}^f$ for which
		$$\sup_{(W(\cdot),\bar g)\in{\cal W}\times\RR_+^J}L_2(\mu,W,\bar g)=L_2(\mu,W,{\bf 0})=\int_{{\bf V}\times{\bf B}}\bar C_0(x,b)\mu^\pi(dx\times db)={\cal V}_0(x_0,\pi)<\infty.$$
		Here $\mu^\pi$ is the restriction of the occupation measure on ${\bf V}\times{\bf B}$, and ${\bf 0}\in\RR_+^J$ is the zero vector.
	\end{remark}
	
	The {\sl dual convex program} is 
	\begin{equation}\label{e25}
		\mbox{Maximize over }~(W(\cdot),\bar g)\in{\cal W}\times\RR_+^J~:~~\inf_{\mu\in{\cal D}^f} L_2(\mu,W,\bar g).
	\end{equation}
	
	One can solve the dual program (\ref{e25}) in two steps: firstly, for a fixed $\bar g\in\RR_+^J$, solve problem
	\begin{equation}\label{e26}
		\inf_{\mu\in{\cal D}^f} L_2(\mu,W,\bar g)\to \sup_{W(\cdot)\in{\cal W}};
	\end{equation}
	and after that  maximize $\sup_{W(\cdot)\in{\cal W}} \inf_{\mu\in{\cal D}^f} L_2(\mu,W,\bar g)$ with respect to $\bar g\in\RR_+^J$.
	Under Condition \ref{con5}(b),
	ignoring the constant $\sum_{j=1}^J g_jd_j$, problem (\ref{e26}) can be rewritten as follows:
	\begin{eqnarray}\label{e27}
		\mbox{Maximize}&:& W(x_0)
		\mbox{ over } W(\cdot)\in{\cal W} \\
		\mbox{subject to}&:&
		\bar C_0(x,b)+\sum_{j=1}^J g_j\bar C_j(x,b)+\int_{\bf V} W(y) Q(dy|x,b)-W(x)\ge 0,\label{e28}\\
		&&~~~~~~~~~~~(x,b)\in{\bf V}\times{\bf B}.\nonumber
	\end{eqnarray}
	Indeed, if inequality (\ref{e28}) is violated, that is, for some $\varepsilon>0$, the left-hand part is below $-\varepsilon$ for some $(\hat x,\hat b)\in{\bf V}\times{\bf B}$,
	then, for the Dirac measures $\mu_N=N\delta_{(\hat x,\hat b)}(dx\times db)\in{\cal D}^f$,   $\lim_{N\to\infty}L_2(\mu_N,W,\bar g)=-\infty$.
	At the same time, $\inf_{\mu\in{\cal D}^f} L_2(\mu,W,\bar g)$ for functions $W(\cdot)$ satisfying (\ref{e28}) is provided by $\mu\equiv 0$ and equals $W(x_0)-\sum_{j=1}^J g_jd_j>-\infty$.
	
	\begin{theorem}\label{t5}
		Suppose Conditions  \ref{ConstrainedPPZcondition05}, \ref{con4}  and \ref{con5} are satisfied and the MDP model is positive and semicontinuous. Then the following assertions hold.
		\begin{itemize}
			\item[(a)] For a fixed $\bar g\in\RR_+^J$, the Bellman function
			\begin{eqnarray}\label{e31}
				W^*_{\bar g}(x)&:=&\inf_\pi\EE^\pi_x\left[\sum_{i=1}^\infty\left(\bar C_0(X_{i-1},B_i)+\sum_{j=1}^J g_j\bar C_j(X_{i-1},B_i)\right)\right]\\
				&=&\inf_\pi\left\{{\cal V}_0(x,\pi)+\sum_{j=1}^J g_j{\cal V}_j(x,\pi)\right\},
				~~~x\in{\bf V}\nonumber
			\end{eqnarray}
			is bounded, non-negative and lower semicontinuous, and 
			solves problem (\ref{e27}-\ref{e28}). Here one can certainly restrict to such strategies $\pi$ that $\pi_i(db|h_{i-1})=\delta_{f^*(x_{i-1})}(db)$ if $x_{i-1}$, the last component of $h_{i-1}$ is in ${\bf V}^c$, so that 
			$$W^*_{\bar g}(x)=\inf_\pi\EE^\pi_x\left[\sum_{i=1}^\infty\left(\bar C_0(X_{i-1},B_i)+\sum_{j=1}^J g_j\bar C_j(X_{i-1},B_i)\right)\right]\equiv 0,~~~x\in{\bf V}^c.$$
			\item[(b)] The optimal value of problem (\ref{e27}-\ref{e28}), $W^*_{\bar g}(x_0)$ is such that $W^*_{\bar g}(x_0)-\sum_{j=1}^J g_jd_j$ coincides with the concave function $h(\bar g)$ as in Theorem \ref{prop2}(a). Thus, an optimal value $\bar g^*$ providing
			$$\max_{\bar g\in\RR_+^J}\left\{\max_{W(\cdot)\in{\cal W}} \inf_{\mu\in{\cal D}^f} L_2(\mu,W,\bar g)\right\}=\max_{\bar g\in\RR_+^J} h(\bar g)$$
			exists.
			\item[(c)] The pair $(W^*_{\bar g^*},\bar g^*)$ is a solution to the dual program (\ref{e25}).
			\item[(d)] There is no duality gap:
			$$\inf_{\mu\in{\cal D}^f}\sup_{(W(\cdot),\bar g)\in{\cal W}\times\RR_+^J} L_2(\mu,W,\bar g)=\sup_{(W(\cdot),\bar g)\in{\cal W}\times\RR_+^J}\inf_{\mu\in{\cal D}^f}L_2(\mu,W,\bar g).$$
		\end{itemize}
	\end{theorem}
	
	\begin{lemma}\label{l7}
		Suppose Conditions  \ref{ConstrainedPPZcondition05} and \ref{con5} are  satisfied, the MDP model is positive and semicontinuous, and $\bar g\in\RR_+^J$ is arbitrarily fixed. Let a function $W(\cdot)\in{\cal W}$ be such that
		\begin{eqnarray}
			W(x) &=& \inf_{b\in{\bf B}}\left\{\bar C_0(x,b)+\sum_{j=1}^J g_j\bar C_j(x,b)+\int_{\bf V} W(y) Q(dy|x,b)\right\}\label{e29}\\
			&=& \bar C_0(x,f(x))+\sum_{j=1}^J g_j\bar C_j(x,f(x))+\int_{\bf V} W(y) Q(dy|x,f(x)),~~~x\in{\bf V}\nonumber
		\end{eqnarray}
		for a measurable mapping $f:~{\bf V}\to{\bf B}$.
		
		Then $W(\cdot)=W^*_{\bar g}(\cdot)$ is the Bellman function and the deterministic stationary strategy $f$ (supplemented by $f(x):=f^*(x)$ for $x\in{\bf V}^c$, as usual: see Definition \ref{d3}) is uniformly optimal in the sense that
		$${\cal V}_0(x,f)+\sum_{j=1}^J g_j{\cal V}_j(x,f)= W^*_{\bar g}(x),~~~x\in{\bf X}_\Delta.$$
	\end{lemma}
	
	\section{Inventory model}\label{sec7}
	Consider the following standard inventory model (Ch. 13.3 of \cite{taha}) with instantaneous order replenishment.
	$L$ is the maximal capacity which is finite, but large enough.
	Let
	\begin{itemize}
		\item $a\ge 0$ be the order quantity, number of units of product (action or impulse);
		\item $D>0$ be the demand rate, units per unit time;
		\item $K>0$ be the setup cost associated with the placement of an order, dollars per order;
		\item $H>0$ be the holding cost, dollars per inventory unit per unit time.
	\end{itemize}
	Suppose the profit of selling one unit equals one dollar. To make the model positive,
	we accept that the cost rate (excluding the holding expenses) is zero when the inventory level is positive, and equals $D$ when it is zero: the unsatisfied demand leads to the loss of one dollar per unit. In other words, $D$ is the shortage cost rate. The state of the system $x\ge 0$ equals the inventory level.
	
	The state space is ${\bf X}=[0,L]$ and the action space is ${\bf A}=[0,L]$. The flow is given by
	$$\phi(x,t)=(x-Dt)\cdot\II\{x-Dt\ge 0\},$$
	and $l(x,a)=\min\{x+a,L\}$. The single constraint at level $d>0$ comes from the holding cost, $\alpha>0$ is the discount factor as usual, and $x_0=0$ is the initial state. The main objective ${\cal V}_0(x_0,\pi)$ comes from  the setup and shortage costs.
	Therefore,
	$$\begin{array}{ll}C^g_0(x)=D\cdot\II\{x=0\}, &~~C^I_0(x,a)=K;\\
		C^g_1(x)=Hx, &~~C^I_1(x,a)=0;\end{array}$$
	and, for $x\in{\bf X}$,
	\begin{eqnarray*}
		\bar C_0(x,(\theta,a))&=&\left\{\begin{array}{ll}
			e^{-\alpha\theta}K, & \mbox{ if } \theta\le \frac{x}{D};\\
			\int_{x/D}^\theta e^{-\alpha t}D~dt+e^{-\alpha\theta}K=\frac{D}{\alpha}(e^{-\frac{\alpha x}{D}} -e^{-\alpha\theta})+e^{-\alpha\theta}K, & \mbox{ if } \theta\ge \frac{x}{D};\end{array}\right.\\
		\bar C_1(x,(\theta,a))&=& \int_0^{\min\{\theta,x/D\}} e^{-\alpha t} H\cdot(x-Dt)dt\\
		&=& \left\{\begin{array}{ll}
			\frac{Hx}{\alpha}(1-e^{-\alpha\theta})
			+HD\left[e^{-\alpha\theta}\left(\frac{\theta}{\alpha}+\frac{1}{\alpha^2}\right)-\frac{1}{\alpha^2}\right], & \mbox{ if } \theta\le\frac{x}{D};\\
			\frac{Hx}{\alpha}
			+\frac{HD}{\alpha^2}\left(e^{-\frac{\alpha x}{D}}-1\right), & \mbox{ if } \theta\ge\frac{x}{D};\end{array}\right.\\
		Q(dy|x,(\theta,a))&=&
		=\left\{\begin{array}{ll}
			e^{-\alpha\theta} \delta_{(\min\{x-D\theta+a,L\})}(dy)\\
			+(1-e^{-\alpha\theta})\delta_\Delta(dy), & \mbox{ if } \theta\le\frac{x}{D},~x\ne\Delta;\\
			e^{-\alpha\theta} \delta_{a}(dy)+(1-e^{-\alpha\theta})\delta_\Delta(dy), & \mbox{ if } \theta\ge\frac{x}{D},~x\ne\Delta;\\
			\delta_\Delta(dy), & \mbox{ if } x=\Delta \mbox{ (or $\theta=\infty$)}.
		\end{array}\right.
	\end{eqnarray*}
	
	Note that the cost rate $C^g_0(\cdot)$ is {\sl upper} semicontinuous, but the cost function $\bar C_0(x,(\theta,a))$ is obviously continuous, and the MDP under study is positive and semicontinuous. Conditions  \ref{con1}, \ref{ConstrainedPPZcondition05},  \ref{con5} and the Slater condition \ref{con4}  are obviously  satisfied: for the deterministic stationary  strategy $f^0(x):=(\infty,\hat a)$ with an arbitrary $\hat a\in{\bf A}$, ${\cal V}_0(x_0,f^0)=\frac{D}{\alpha}$ and ${\cal V}_1(x_0,f^0)=0<d$ because $\bar C_0(x_0,(\infty,\hat a))=\frac{D}{\alpha}$ and $\bar C_1(x_0,(\infty,\hat a))=0$. Recall that $x_0=0$. Since $J=1$, we omit the bar and write $g$ and $g^*$ for $\bar g$ and $\bar g^*$. Finally, according to the comments below Condition  \ref{con5}, ${\bf V}={\bf X}$. Actions in the absorbing costless state $\Delta$ are fixed arbitrarily.
	
We assume that the constant $L$, the capacity of the storage, is  large enough. The minimal value can be calculated, using the other parameters, as follows.

For each $g>0$, let $a_g>0$ be the unique positive solution to equation 
		\begin{equation}\label{eqn1}
	\frac{\alpha K}{D}+\frac{gH}{\alpha}+\frac{gH}{D} a=\frac{gH}{\alpha} e^{\frac{\alpha a}{D}}.
\end{equation}

	\begin{lemma}\label{l9} Suppose $g>0$ and  $a_g$ comes from equation (\ref{eqn1}). Then the following assertions hold.
	\begin{itemize}
		\item[(a)] The function $a_g$ is differentiable and decreases with $g$.
		\item[(b)] For each $a>0$ there is a unique $g>0$ such that $a=a_g$, so that positive $g$ and $a_g$ are in 1-1 correspondence, and $\lim_{g\downarrow 0} a_g=\infty$, $\lim_{g\to\infty} a_g=0$.
		\item[(c)] The product $ga_g$ increases with $g$ from zero (when $g\downarrow 0$) to $\infty$ (when $g\to\infty$).
	\end{itemize}
\end{lemma}

According to Lemma \ref{l9}(c), there exists the unique value $g_c>0$ such that 
\begin{equation}\label{e11}
\alpha K+Hg_c a_{g_c}=D.
\end{equation}

Let $a^*$ be the unique positive solution to equation
	\begin{equation}\label{e15}
	\frac{Ha}{\alpha}\cdot \frac{e^{\alpha\frac{a}{D}}}{e^{\alpha\frac{a}{D}}-1}-\frac{DH}{\alpha^2}-d=0,
\end{equation}
so that $a^*=a_{\hat g}$, where, by (\ref{eqn1}),
	$$\hat g=\frac{\alpha^2 K}{DH e^{\alpha\frac{a^*}{D}}-DH-\alpha Ha^*}>0.$$
The strict inequality holds because positive $g$ and $a_g$ are in 1-1 correspondence.
The solvability of equation (\ref{e15}) will be established during the proof of Theorem \ref{t8}.
	
Now it is sufficient to  require that
$$L>\max\{a_{g_c},a^*\}.$$

Since the function $a_g$ is continuous (see Lemma \ref{l9}(a)), there is $\varepsilon>0$ such that
$$\hat g-\varepsilon>0,~g_c-\varepsilon>0~\mbox{ and } a_{\hat g-\varepsilon}<L,~a_{g_c-\varepsilon}<L,$$
and we fix
$g_{min}:=\min\{\hat g,g_c\}-\varepsilon>0$.
It will be clear  that the values of $g$ smaller than $g_{min}$ are of no interest.
	
	\begin{lemma}\label{l10}
		Suppose $\alpha K\ge D$. Then at $g=0$
		$$W^*_0(x):=\frac{D}{\alpha}e^{-\frac{\alpha x}{D}}=\inf_\pi{\cal V}_0(x,\pi),~~x\in{\bf X},$$
		and $f^0(x)\equiv\left(\infty,\hat a\right)$ for all $x\in {\bf X}$, with an arbitrary $\hat a\in{\bf A}$,
		is the  single (up to the value of $\hat a$) uniformly optimal deterministic stationary strategy in the problem ${\cal V}_0(x,\pi)\to\inf_\pi$, i.e.,
		${\cal V}_0(x,f^0)=W^*_0(x)=\inf_\pi {\cal V}_0(x,\pi),~x\in{\bf X}_\Delta$.
		In words, one should never place any orders (because the setup cost is too large). Of course, as usual, $W^*_0(\Delta)=0$.
	\end{lemma}
	
	In the framework of Lemma \ref{l10}, the deterministic stationary strategy$f^0(x)=(\infty,\hat a)$ is obviously optimal also in the problem ${\cal V}_1(x,\pi)\to\inf_\pi$ because no new orders are placed and the holding cost is minimal possible. Thus, $f^0$ is the solution to the constrained problem (\ref{PZZeqn02}) for any $d\ge 0$: recall that $x_0=0$ and ${\cal V}_1(0,f^0)=0$.
	
	In what follows, we assume that $\alpha K<D$.
	
	\begin{theorem}\label{t6} Assume that $\alpha K<D$, $g\ge g_{min}$ is fixed and $a_g$ comes from equation (\ref{eqn1}). Note, $L>a_g$ by Lemma \ref{l9}(a).
		
	 Then the following assertions hold.
		\begin{itemize}
			\item[(a)] If
			\begin{equation}\label{eqn2}
				\alpha K+Hga_g<D,
			\end{equation}
			then the Bellman function (\ref{e31}) has the form
			\begin{equation}\label{eqn3}
				W^*_g(x)=\frac{xgH}{\alpha}-\frac{DgH}{\alpha^2}+\frac{DgH}{\alpha^2}e^{\alpha\frac{a_g-x}{D}},~~x\in{\bf X},
			\end{equation}
			and $f(x)=\left(\frac{x}{D},a_g\right)$, $x\in {\bf X}$, is the single uniformly optimal deterministic stationary strategy, i.e.,
			\begin{equation}\label{e34}
				{\cal V}_0(x,f)+g{\cal V}_1(x,f)=W^*_g(x),~~~~~x\in{\bf X}.
			\end{equation}
			In words, one has to wait until the inventory  reaches zero level, and order $a_g$ units.
			\item[(b)] If
			\begin{equation}\label{eqn4}
				\alpha K+Hga_g>D,
			\end{equation}
			then  the Bellman function (\ref{e31}) has the form
			\begin{equation}\label{eqn5}
				W^*_g(x)=\frac{xgH}{\alpha}-\frac{DgH}{\alpha^2}+\left(\frac{DgH}{\alpha^2}+\frac{D}{\alpha}\right)e^{-\frac{\alpha x}{D}},~~x\in{\bf X},
			\end{equation}
			and $f(x)=\left(\infty,\hat a\right)$, $x\in {\bf X}$, with an arbitrary $\hat a\in{\bf A}$
			is the single (up to the value of $\hat a$) uniformly optimal deterministic stationary strategy in the sense of (\ref{e34}).
			In words, one should never place any orders because the setup and  holding costs are too large as well as the Lagrange multiplier $g$.
			\item[(c)] If $\alpha K+Hga_g=D$, then functions (\ref{eqn3}) and (\ref{eqn5}) coincide, and all the deterministic stationary strategies of the form $f(x)=\left(\theta\ge \frac{x}{D},a_g\right)$, $x\in {\bf X}$ are uniformly optimal in the sense of (\ref{e34}).
		\end{itemize}
	\end{theorem}
	
All the statements of Theorem \ref{t6} remain valid if $g>0$ is arbitrary and $L\ge a_g$.
	
	Suppose $g>0$ is fixed at a small level: $\sqrt{2KDHg}<D$, and compute $\lim_{\alpha\downarrow 0} a_g$. To do so, one can take three terms in the Taylor expansion of $e^{\frac{\alpha a}{D}}=1+\frac{\alpha a}{D}+\frac{\alpha^2 a^2}{2D^2}+\ldots$. Now equation (\ref{eqn1}) takes the form
	$$\alpha^2 K+DgH+\alpha gH a=DgH(1+\frac{\alpha a}{D}+\frac{\alpha^2 a^2}{2D^2}+\ldots)$$
	leading to 
	$$\breve a:=\lim_{\alpha\downarrow 0} a_g=\sqrt{\frac{2KD}{gH}}.$$
	The obtained expression is consistent with the standard formula for Economic Order Quantity (EOQ), see p.508 of \cite{taha}. Note that the product $gH$ plays the role of the holding cost denoted just as $h$ in \cite{taha}. Since $Hg\breve a=\sqrt{2KDHg}$, $\alpha K+Hga_g<D$ for all small enough $\alpha>0$, so that we are in the framework of Theorem \ref{t6}(a). The deterministic stationary strategy $f(x)=(\frac{x}{D},a_g)$ is uniformly optimal  for all small $\alpha>0$ by Theorem \ref{t6}(a), and $a_g$ approaches EOQ when $\alpha\downarrow 0$. In Subsection 13.3.1 of \cite{taha}, the optimal control problem was undiscounted, and the objective was the long-run average setup and holding cost, called Total Cost per Unit time (TCU).  Its minimal value is $gH\breve a=\sqrt{2KDHg}$. Note that the shortage cost was not considered in Subsection 13.3.1 of \cite{taha}.
	If we accept that $D$ is the cost rate associated with unsatisfied demand, then one has to compare $\sqrt{2KDHg}$ and $D$: in case $\sqrt{2KDHg}>D$, it is optimal not to place any orders because the setup cost $K$ and the holding cost $gH$ are too large. This was the reason to introduce inequality $\sqrt{2KDHg}<D$ when looking at the limit when $\alpha\downarrow 0$ and comparing the results with \cite{taha}.
	
		Let us  introduce the {\sl critical} value of the constraint $d$: 
		$$d_c:=\frac{Ha_{g_c}e^{\alpha\frac{a_{g_c}}{D}}}{\alpha(e^{\alpha\frac{a_{g_c}}{D}}-1)}-\frac{DH}{\alpha^2}=\frac{\frac{Ha_{g_c}}{\alpha}+\frac{HD}{\alpha^2}\left(e^{-\alpha\frac{a_{g_c}}{D}}-1\right)}{1-e^{-\alpha\frac{a_{g_c}}{D}}}>0.
		$$
		For the strict inequality, see the expression for $\bar C_1(x,(\theta,a))$ at $x=a_{g_c}$ and $\theta\ge \frac{a_{g_c}}{D}$.
		Recall that $g_c$, below called the {\sl critical value} of the Lagrange multiplier, comes from equation (\ref{e11}).

	\begin{theorem}\label{t8} Assume that $\alpha K<D$.
		\begin{itemize}
			\item[(a)] If $d\le d_c$, then $g^*=g_c$ and the optimal (and feasible) solution to problem (\ref{PZZeqn02})  is given by the deterministic stationary strategy $f^*(x):=(\theta^*(x)=\frac{x}{D}+\tau^*,a_{g_c})$, where
			\begin{eqnarray*}
				\tau^*&:=&\frac{1}{\alpha}\ln\left[ \frac{Ha_{g_c}}{\alpha d}+\frac{HD}{\alpha^2 d}\left( e^{-\alpha\frac{a_{g_c}}{D}}-1\right)+e^{-\alpha\frac{ a_{g_c}}{D}}\right]\\
				&=&\frac{1}{\alpha}\ln\left[\left(1-e^{-\alpha\frac{a_{g_c}}{D}}\right)\frac{d_c}{d}+e^{-\alpha\frac{ a_{g_c}}{D}}\right].
			\end{eqnarray*}
			Moreover, ${\cal V}_1(0,f^*)=d$ and ${\cal V}_0(0,f^*)=\frac{D}{\alpha}-g^*d$.
			Note that $\tau^*\ge 0$ when $d\le d_c$. 
			
			In words, at the very beginning (when $x_0=0$) and as soon as the inventory level reaches zero, one has to wait for $\tau^*$ time units and order $a_{g_c}$ units of product: see Figure \ref{f1}.
			\item[(b)] If $d>d_c$, then
			$$g^*=\hat g=\frac{\alpha^2 K}{DH e^{\alpha\frac{a^*}{D}}-DH-\alpha Ha^*}<g_c,$$
			where $a^*=a_{g^*}$ is the unique  positive solution to equation (\ref{e15}).
			The deterministic stationary strategy $f^*(x):=(\theta^*(x)=\frac{x}{D},a^*)$ is the optimal (and feasible) solution to problem (\ref{PZZeqn02}).
			
			Moreover, ${\cal V}_1(0,f^*)=d$ and ${\cal V}_0(0,f^*)=\frac{Dg^*H}{\alpha^2}\left(e^{\alpha\frac{a^*}{D}}-1\right)-g^*d$.
			
			In words, at the very beginning (when $x_0=0$) and as soon as the inventory level reaches zero, one has to  order immediately $a^*$ units of product.
		\end{itemize}
	\end{theorem}
	
	\begin{figure}[htbp]
		\begin{center}
			\includegraphics[width=8cm]{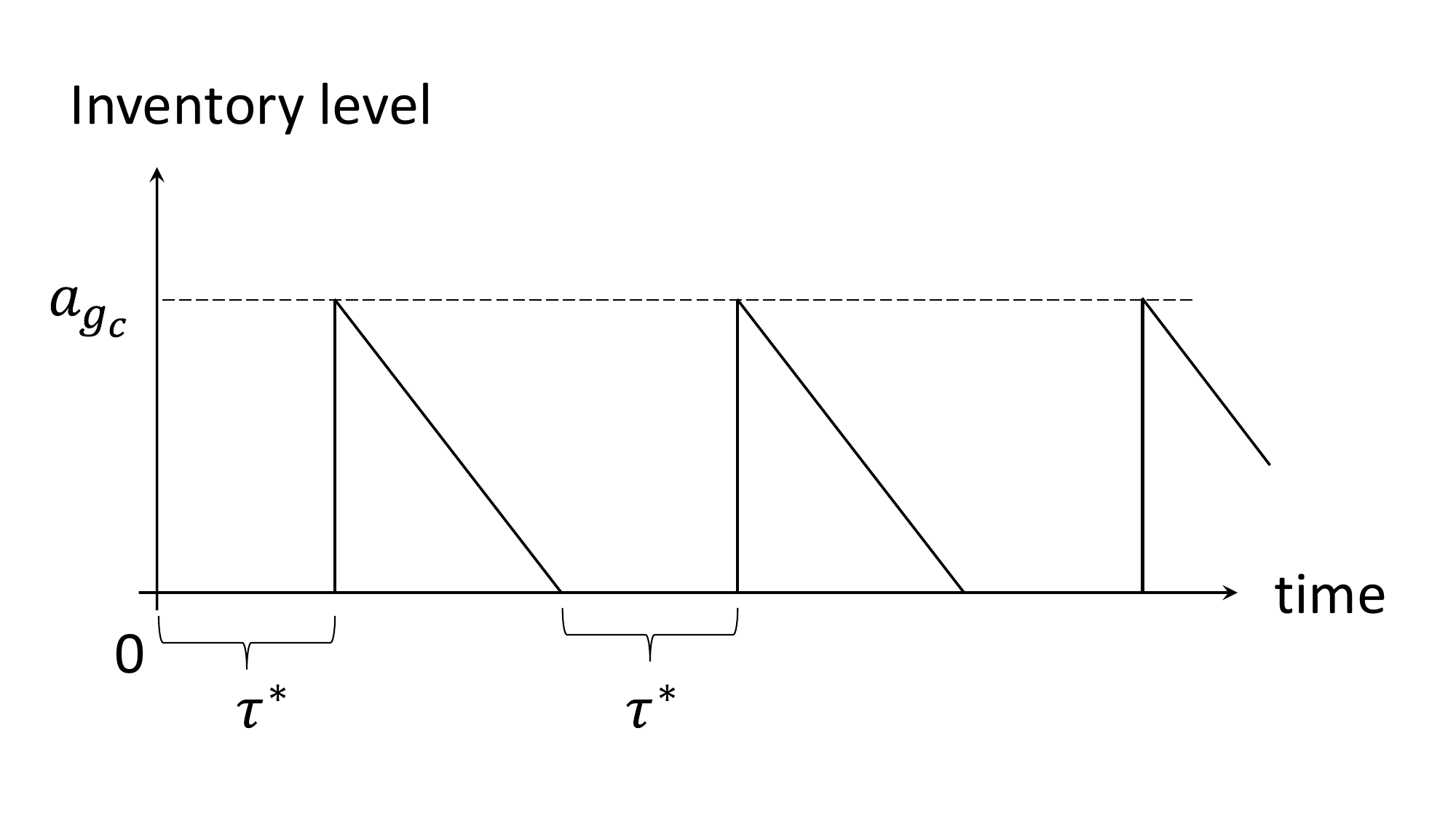}
			\caption{Optimal inventory strategy: Theorem \ref{t8}(a). }\label{f1}
		\end{center}
	\end{figure}
	
	\section{Conclusion}\label{conc}
	The MDP approach to
	the constrained optimal impulse control problem with discounting is developed. Using the linear programming approach, two pairs of convex programs in the space of occupation measures are studied in depth. Under appropriate conditions, all the programs have optimal solutions, and necessary and sufficient conditions of optimality are presented.
	
	The theoretical results are illustrated by the standard inventory model. Since the model is with the constraint on the holding cost, in some cases it is optimal to keep for a while the stock at zero level before ordering the product.

	\section*{Appendix}\label{secap}
	\underline{Proof of Proposition \ref{pr1}.} $\bf B$ is compact (in the product topology) as the product of compacts.
	
	If $u:~{\bf X}_\Delta\mapsto\RR$ is a bounded continuous function, then
	$$\int_{{\bf X}_\Delta} u(y)Q(dy|x,b=(\theta,a))=\left\{\begin{array}{ll}
		e^{-\alpha\theta}u(l(\phi(x,\theta),a))\\
		+(1-e^{-\alpha\theta})u(\Delta) & \mbox{ if } x\ne\Delta,~\theta\ne+\infty;\\
		u(\Delta) & \mbox{ otherwise}\end{array}\right.$$
	is a continuous function on ${\bf X}_\Delta\times{\bf B}$. Thus $Q$ is a continuous stochastic kernel.
	
	Let us show that the cost functions $\bar C_j(\cdot,\cdot)$ are lower semicontinuous.
	The case of $x=\Delta$, when $\bar C_j(x,b)=0$, is trivial for the isolated point $\Delta$.
	A non-negative function is lower semicontinuous if and only if it is the limit of an increasing sequence of bounded non-negative continuous functions \cite[Lemma 7.14]{Bertsekas:1978}. Therefore, $C^{g,m}_j(\cdot)\uparrow C^g_j(\cdot)$ as $m\to\infty$ for non-negative, bounded and continuous functions $C^{g,m}_j(\cdot)$, because the functions $C^j_j(\cdot)$ are non-negative and lower semicontinuous. Using the monotone convergence theorem, we see that, for each $j=0,1,\ldots,J$, the function
	$${\bf X}\times{\bf B}\ni(x,b=(\theta,a)) \to  \int_0^\theta e^{-\alpha t}  C^g_j(\phi(x,t))dt
	= \lim_{m\to\infty}  \int_0^\theta e^{-\alpha t}  C^{g,m}_j(\phi(x,t))dt$$
	is also non-negative and lower semicontinuous.  The function $e^{-\alpha\theta} C^I_j(\phi(x,\theta),a)$ is lower semicontinuous on ${\bf X}\times[0,\infty)\times{\bf A}$ and
	$$\liminf_{n\to\infty} e^{-\alpha\theta_n}C^I_j(\phi(x_n,\theta_n),a_n)\ge 0=e^{-\alpha\infty}C^I_j(\phi(x,\theta),a)$$
	if $(x_n,(\theta_n,a_n))\to (x,(+\infty,a))$.
	\hfill$\square$\bigskip
	
	\underline{Proof of Lemma \ref{l12}.} Introduce the space $U:=\{(P,\bar\nu)\}$, where $P=\{P_1,P_2,\ldots\}$ is a measurable partition of $\bf V$ and $\bar\nu=\{\nu_1,\nu_2,\ldots\}$ is a sequence of signed finite measures: $\nu_i$ is defined on $P_i\times{\bf B}$, $i=1,2,\ldots$. Introduce the following operations: for $k\in\RR$, $k\cdot(P,\bar\nu):=(P,k\bar\nu)$ and $(P^1,\bar\nu^1)+(P^2,\bar\nu^2)$ is defined by $(P^3,\bar\nu^3)$, where $P^3=\{P^1_i\cap P^2_j,~i,j=1,2,\ldots\}$ and $\bar\nu^3=\{\nu^3_{i,j}=\nu^1_i+\nu^2_j,~i,j=1,2,\ldots\}$ with $\nu^3_{i,j}$ being defined on $(P^1_i\cap P^2_j)\times{\bf B}$.  Next, we say that $(P^1,\bar\nu^1)\sim(P^2,\bar\nu^2)$ if $\nu^1_i=\nu^2_j$ on each set $(P^1_i\cap P^2_j)\times{\bf B}$, $i,j=1,2,\ldots$. The described relation $\sim$ is obviously reflective and symmetric. For the transitivity, note that, if $(P^1,\bar\nu^1)\sim(P^2,\bar\nu^2)$ and $(P^2,\bar\nu^2)\sim(P^3,\bar\nu^3)$, then $\nu^1_i=\nu^2_j=\nu^3_k$ for each set $P^1_i\cap P^2_j\cap P^3_k$. Hence, $\nu^1_i=\nu^3_k$ on each set $P^1_i\cap P^3_k$. The factor (quotient)  space with respect to the described equivalence relation $\sim$ is $[U]$. The operations of scalar multiplication and addition on $[{\bf U}]$ are defined by $k\cdot[(P,\bar\nu)]:=[k\cdot(P,\bar\nu)]$ and $[(P,\bar\nu)]+[(Q,\bar\eta)]:=[(P,\bar\nu)+(Q,\bar\eta)]$ for each $[(P,\bar\nu)],[(Q,\bar\eta)]\in[{\bf U}]$. Then $[{\bf U}]$ is a vector space.	Each element $(Q,\bar\eta)$ of $[(P,\bar\nu)]\in[U]$ with
	$\bar\nu=\{\nu_1,\nu_2,\ldots\}$ and $\bar\eta=\{\eta_1,\eta_2,\ldots\}$, where $\nu_1,\nu_2,\ldots,\eta_1,\eta_2\ldots$ are finite
	 non-negative measures is such that the partition $Q$ may be different from the partition $P$, but each $\nu_i$ and $\eta_j$ are just restrictions on $P_i\times{\bf B}$ and $Q_j\times{\bf B}$ of a common non-negative measure $\mu$ on ${\bf V}\times{\bf B}$. Thus, $[(P,\bar\nu)]$ is identified with the measure $\mu$ with the $\sigma$-finite marginal on $\bf V$, i.e.  with property (\ref{enum2}). Thus, the set $\cal D$ can be considered as a convex subset of $[U]$ because equation (\ref{e9}) and condition (\ref{e13p}) are linear w.r.t. $\mu$.
	\hfill$\square$\bigskip
	
	\underline{Proof of Theorem \ref{t1}.} (a) According to Remark \ref{rem71},
	val(\ref{PZZeqn02})$\ge$val(\ref{enum4p}).
	In particular, by Condition \ref{ConstrainedPPZcondition05}, val(\ref{e13pp}-\ref{e14pp})$=$val(\ref{enum4p}) is finite, and one needs to consider only the measures $\mu$ with $\displaystyle \int_{{\bf V}\times{\bf B}}\bar C_0(x,b)\mu(dx\times db)<\infty$. If $\mu$ is such a feasible measure in the  program  (\ref{e13pp})-(\ref{e14pp}), then
	$${\cal V}_j(x_0,\pi^\mu)\le \int_{{\bf V}\times{\bf B}} \bar C_j(x,b)\mu(dx\times db),~~j=0,1,\ldots,J$$
	by Lemma \ref{l3}. Therefore, the induced strategy $\pi^\mu$ is feasible in problem (\ref{PZZeqn02}) and 
	${\cal V}_0(x_0,\pi^\mu)\le \int_{{\bf V}\times{\bf B}} \bar C_0(x,b)\mu(dx\times db)$
	leading to the required inequality val(\ref{PZZeqn02})$\le$val(\ref{enum4p})$=$val(\ref{e13pp}-\ref{e14pp}).
	
	(b) Let $\pi^*\in\Pi$ be a feasible and optimal strategy in problem (\ref{PZZeqn02}). Then the occupation measure $\mu^{\pi^*}$ is feasible in the primal program (\ref{e13pp})-(\ref{e14pp}) according to Corollary \ref{cor1} and explanations below it. Moreover,
	\begin{eqnarray*}
		{\cal V}_0(x_0,\pi^*)&=&\int_{{\bf V}\times{\bf B}} \bar C_0(x,b)\mu^{\pi^*}(dx\times db)=\mbox{val(\ref{PZZeqn02})}=\mbox{val(\ref{enum4p})},
	\end{eqnarray*}
	and thus the measure $\mu^{\pi^*}$ is feasible and optimal in the primal program (\ref{e13pp})-(\ref{e14pp}).
	
	(c) Let measure $\mu^*$ be feasible and optimal in the primal  program (\ref{e13pp})-(\ref{e14pp}). Then \linebreak $ \int_{{\bf V}\times{\bf B}} \bar C_0(x,b)\mu^*(dx\times db)<\infty$ because, as was shown, val(\ref{enum4p})$=$val(\ref{e13pp}-\ref{e14pp}) is finite.
	The induced strategy $\pi^{\mu^*}$ is feasible in problem (\ref{PZZeqn02}) by Lemma \ref{l3}. Moreover,
	\begin{eqnarray*}
		{\cal V}_0(x_0,\pi^{\mu^*})&\le &\int_{{\bf V}\times{\bf B}} \bar C_0(x,b)\mu^*(dx\times db)=\mbox{val(\ref{enum4p})}=\mbox{val(\ref{PZZeqn02}).}
	\end{eqnarray*}
	Thus the strategy $\pi^{\mu^*}$ is feasible and optimal in problem (\ref{PZZeqn02}). Obiously, $\pi^{\mu^*}\in\Pi$ according to Definitions \ref{d3} and \ref{d4}. \hfill$\square$\bigskip
	
	\par\noindent\textit{Proof of Lemma \ref{l4}.} Consider a strategy $\pi$ such that $\mu^\pi({\bf X}\times{\bf B})=\infty.$ Recall that $\mu^\pi$ satisfies 
	\begin{eqnarray*}
		\mu^\pi(dx\times {\bf B})=\delta_{x_0}(dx)+\int_{\bf X}\int_{[0,\infty)}\int_{\bf A}e^{-\alpha\theta}\delta_{l(\phi((y,\theta),a))}(dx)\mu^{\pi}(dy\times d\theta\times da)
	\end{eqnarray*}
	on $({\bf X},{\cal B}({\bf X})).$ Since $\mu^\pi({\bf X}\times{\bf B})=\infty$, we have $$\infty=1+\int_{\bf X}\int_{[0,\infty)}\int_{\bf A}e^{-\alpha\theta} \mu^{\pi}(dy\times d\theta\times da)=1+ \int_{[0,\infty)}e^{-\alpha\theta} \mu^{\pi}({\bf X}\times d\theta\times {\bf A}).$$ Now, 
	\begin{eqnarray*}
		{\cal V}_0(x_0,\pi)&=& \int_{{\bf X}_\Delta\times [0,\infty]\times{\bf A}} \bar{C}_0(x,(\theta,a))\mu^\pi(dx\times d\theta\times da)\\
		&\ge& \int_{[0,\infty)} e^{-\alpha\theta} \delta\mu^\pi({\bf X}\times d\theta\times {\bf A})=\infty,
	\end{eqnarray*}
	as required. $\hfill\Box$
	\bigskip
	
	\underline{Proof of Theorem \ref{t5}.} (a) Let $\bar g\in\RR_+^J$ be arbitrarily fixed. Under Condition \ref{con5}(b), $0\le W^*_{\bar g}(x)\le(J+1){\cal C}\int_0^\infty e^{-\alpha t}dt$: apply the strategy $\pi_i(db|h_{i-1})=\delta_\infty(d\theta)\delta_{\hat a}(da)$ with an arbitrarily fixed $\hat a\in{\bf A}$. Thus, $W^*_{\bar g}(\cdot)$ is a bounded non-negative lower semicontinuous function on $\bf V$, and there exists an optimal deterministic stationary strategy $\hat f$ providing the infimum in (\ref{e31}):  see Corollary 9.17.2 of \cite{Bertsekas:1978}. Without loss of generality, we fix $\hat f(x)=f^*(x)$ for $x\in{\bf V}^c$. (Recall equalities (\ref{e5}).)
	The corresponding occupation measure $\mu_{\bar g}:=\mu^{\hat f}$, as usual corresponding to the fixed initial state $x_0$, is finite on ${\bf V}\times{\bf B}$ by Lemma \ref{l4} and satisfies equality (\ref{e22}). (See (\ref{enum1}.)
	
	For the pair of problems (\ref{e26}) and 
	\begin{equation}\label{e32}
		\sup_{W(\cdot)\in{\cal W}} L_2(\mu,W,\bar g)\to\inf_{\mu\in{\cal D}^f},
	\end{equation}
	it is known that val(\ref{e32})$\ge$val(\ref{e26}).
	(See, e.g., (1.8) in \cite{rock}.) Arguing similarly to Remark \ref{r2}, we see that problem (\ref{e32})  is equivalent to the following one:
	\begin{eqnarray*}
		&&\mbox{Minimize over }\mu\in{\cal D}^f~\int_{{\bf V}\times{\bf B}} \left(\bar C_0(x,b)+\sum_{j=1}^J g_j\bar C_j(x,b)\right)\mu(dx\times db) - \sum_{j=1}^J g_jd_j\\
		&&\mbox{subject to (\ref{e22})}.
	\end{eqnarray*}
	Therefore, since the measure $\mu_{\bar g}$ is in ${\cal D}^f$ and satisfies equation (\ref{e22}),
	\begin{equation}\label{enum22}
		\mbox{val(\ref{e32})}\le \int_{{\bf V}\times{\bf B}} \left(\bar C_0(x,b)+\sum_{j=1}^J g_j\bar C_j(x,b)\right)\mu_{\bar g}(dx\times db)-\sum_{j=1}^Jg_jd_j
	\end{equation}
	and, since problem  (\ref{e26}) is equivalent to (\ref{e27}-\ref{e28}), val(\ref{e26})$\ge W^*_{\bar g}(x_0)-\sum_{j=1}^Jg_jd_j$
	because inequality (\ref{e28}) certainly holds for the bounded Bellman function $W^*_{\bar g}(\cdot)\in{\cal W}$. But
	$$\int_{{\bf V}\times{\bf B}} \left(\bar C_0(x,b)+\sum_{j=1}^J g_j\bar C_j(x,b)\right)\mu_{\bar g}(dx\times db)=W^*_{\bar g}(x_0)$$
	because here the both expressions equal (\ref{e31}) under $x=x_0$. Thus, by (\ref{enum22}), val(\ref{e26})$\ge$val(\ref{e32}),
	and the optimal values of problems (\ref{e32}) and (\ref{e26}) coincide and equal
	$W^*_{\bar g}(x_0)-\sum_{j=1}^Jg_jd_j$.
	Hence $W^*_{\bar g}(\cdot)$ solves problem (\ref{e27})-(\ref{e28}) which is equivalent to  problem (\ref{e26}).
	
	(b) This statement follows directly from Remark \ref{rem3}(a).
	
	(c) This  statement follows from parts (a) and (b).
	
	(d) According to parts (b,c),
	$\sup_{(W(\cdot),\bar g)\in{\cal W}\times\RR_+^J} \inf_{\mu\in{\cal D}^f} L_2(\mu,W,\bar g)=\mbox{val(\ref{e25}) }=h(\bar g^*)$.
	By Theorem \ref{prop2}(a),
	$h(\bar g^*)=\mbox{val(\ref{enum4}) }=\mbox{val(\ref{enum4p})}$.
	Finally, under the imposed conditions, val(\ref{enum4p})$=$val(\ref{e13pp}-\ref{e14pp})$=$val(\ref{e21}-\ref{e23})$=$val(\ref{e24})$= \inf_{\mu\in{\cal D}^f}\sup_{(W(\cdot),\bar g)\in{\cal W}\times\RR_+^J} L_2(\mu,W,\bar g)$.
	\hfill$\square$\bigskip
	
	\underline{Proof of Lemma \ref{l7}.} To show that  $W(x)\le{\cal V}_0(x,\pi)+\sum_{j=1}^J g_j{\cal V}_j(x,\pi)$ for all strategies $\pi$ and all $x\in{\bf V}$, it is sufficient to consider only the strategies with ${\cal V}_0(x,\pi)<\infty$. Recall equalities (\ref{e0}) and ${\cal V}_j(x,\pi)=0$ for all $x\in{\bf V}^c$ and all strategies under consideration. For such strategies, $\lim_{i\to\infty}T_i=\infty$ $\PP^\pi_x$-almost surely. Indeed, otherwise, $\PP^\pi_x(\exists T<\infty:~\forall i=1,2,\ldots~T_i\le T)>0$ and
	$${\cal V}_0(x,\pi)\ge \EE^\pi_x\left[\sum_{i=1}^\infty e^{-\alpha T_i}\delta\right]\ge \PP^\pi_x(\exists T<\infty:~\forall i=1,2,\ldots~T_i\le T) \sum_{i=1}^\infty e^{-\alpha T} \delta=\infty.$$
	Therefore, for the strategies of our interest, 
	$$|\EE^\pi_x[W(X_I)]|\le\sup_{y\in{\bf V}} |W(y)|\cdot\EE^\pi_x\left[e^{-\alpha T_I}\right]\to 0 \mbox{ as } I\to\infty,~~~\forall x\in{\bf V}.$$
	Property $\PP^\pi_x(X_I\in{\bf X})=\EE^\pi_x[e^{-\alpha T_I}]$ is here in use. 
	
	Further, for each $x\in{\bf V}$, $I=1,2,\ldots$,
	\begin{eqnarray*}
		0 &\le & \EE^\pi_x\left[\sum_{i=1}^I\left\{\bar C_0(X_{i-1},B_i)+\sum_{j=1}^J g_j\bar C_j(X_{i-1},B_i)-W(X_{i-1})\right.\right.\\
		&&\left.\left.\vphantom{\sum_i^I}+\int_{\bf V} W(y)Q(dy|X_{i-1},B_i)\right\}\right]\\
		&=&\EE^\pi_x\left[\sum_{i=1}^I\left\{\bar C_0(X_{i-1},B_i)+\sum_{j=1}^J g_j\bar C_j(X_{i-1},B_i)\right\}\right]-\EE^\pi_x[W(X_0)]\\
		&&+\EE^\pi_x[W(X_1)]-\EE^\pi_x[W(X_1)]+\EE^\pi_x[W(X_2)]-\ldots+\EE^\pi_x[W(X_I)].
	\end{eqnarray*}
	Note that all the particular terms here are finite.
	After passing to the limit as $I\to\infty$, we obtain the desired inequality
	$W(x)\le {\cal V}_0(x,\pi)+\sum_{j=1}^J g_j{\cal V}_j(x,\pi)$ for all $x\in{\bf V}$.
	
	For the strategy $f$ we have equalities in the previous formulae: 
	$$0= {\cal V}_0(x,f)+\sum_{j=1}^J g_j{\cal V}_j(x,f)-W(x)+\lim_{I\to\infty} \EE^f_x[W(X_I)].$$
	We see that ${\cal V}_0(x,f)<\infty$, and so $\lim_{I\to\infty} \EE^f_x[W(X_I)]=0$. Finally,
	$$W(x)= {\cal V}_0(x,f)+\sum_{j=1}^J g_j{\cal V}_j(x,f),$$
	and the proof is completed. \hfill$\square$\bigskip

	\underline{Proof of Lemma \ref{l9}.} 
	The solvability of equation (\ref{eqn1}) is obvious. Below, we omit the subscript $g$ in $a_g$ for brevity.
	
	(a) Using the well known formula for the derivative of an implicit function, we obtain
\begin{eqnarray}\label{e43}
	\frac{da}{dg} &=& -\frac{\frac{\partial}{\partial g}\left(\frac{gH}{\alpha} e^{\frac{\alpha a}{D}}-\frac{\alpha K}{D}-\frac{gH}{\alpha}-\frac{gH}{D}a\right)}{\frac{\partial}{\partial a}\left(\frac{gH}{\alpha} e^{\frac{\alpha a}{D}}-\frac{\alpha K}{D}-\frac{gH}{\alpha}-\frac{gH}{D}a\right)}\\
	&=&-\frac{\frac{H}{\alpha} e^{\frac{\alpha a}{D}}-\frac{H}{\alpha}-\frac{Ha}{D}}{\frac{gH}{D}e^{\frac{\alpha a}{D}}-\frac{gH}{D}}=-\left[\frac{D}{\alpha g}-\frac{a}{g(e^{\frac{\alpha a}{D}}-1)}\right]=-\frac{\alpha K}{g^2H(e^{\frac{\alpha a}{D}}-1)}<0.\nonumber
\end{eqnarray}
The last equality is by (\ref{eqn1}).
	Both the numerator and the denominator are positive at $a>0$. 
	Therefore, $a_g$ decreases with $g$.
	
	(b) The expression for $g$ follows from the definition (\ref{eqn1})  of $a_g$:
	$$g=\frac{\frac{\alpha K}{D}}{\frac{H}{\alpha} e^{\frac{\alpha a}{D}}-\frac{H}{\alpha}-\frac{H a}{D}}=\frac{\alpha^2K}{HD[e^{\frac{\alpha a}{D}}-1-\frac{\alpha a}{D}]}>0.$$
	Since the positive $g$ and $a_g$ are in 1-1 correspondence and $a_g$ changes monotonically with respect to $g$, the last assertions in item (b) follow from the obvious expressions:
	$$\lim_{a\downarrow 0} g=\infty,~~~\lim_{a\to\infty} g=0.$$
	
	(c)
	$$\frac{d(ag)}{dg}= a+g\frac{da}{dg}=a-g\left[\frac{D}{\alpha g}-\frac{a}{g(e^{\frac{\alpha a}{D}}-1)}\right]
	=\frac{ae^{\frac{\alpha a}{D}}-\frac{D}{\alpha} e^{\frac{\alpha a}{D}}+\frac{D}{\alpha}}{e^{\frac{\alpha a}{D}}-1}.$$
	Both the numerator and the denominator are positive at $a>0$. (Note, the derivative w.r.t. $a$ of the numerator is positive.)
	Therefore, the product $ga_g$ increases with $g$.
	
	According to Item (b), one can firstly study the dependence of the product $ga_g$ on $a_g$:
	$$ga_g=\frac{\alpha^2 K}{HD}\cdot\frac{a_g}{e^{\frac{\alpha a_g}{D}}-1-\frac{\alpha a_g}{D}}\longrightarrow \frac{\alpha^2 K}{HD}\cdot\frac{1}{\frac{\alpha}{D} e^{\frac{\alpha a_g}{D}}-\frac{\alpha}{D}}\longrightarrow
	\left\{\begin{array}{ll} \infty \mbox{ as } a_g\downarrow 0\Longleftrightarrow g\to\infty;\\ 0 \mbox{ as } a_g\to \infty\Longleftrightarrow g\downarrow 0\end{array}\right.$$
	by L'Hopital's rule.
	\hfill$\square$\bigskip
	
	\underline{Proof of Lemma \ref{l10}.} 
According to Lemma \ref{l7}, it is sufficient to check that the presented function $W^*_0(\cdot)$ satisfies the optimality equation (\ref{e29}) with $f=f^0$.

Since $\bar C_0(x,b=(\theta,a))$ does not depend on $a\in{\bf A}$, we firstly compute $\min_{a\in{\bf A}} W^*_0(a)=\frac{D}{\alpha}e^{-\frac{\alpha L}{D}}$ which is achieved at $a=L$. To compute the right-hand part of (\ref{e29}) for $W^*_0(\cdot)$, we firstly minimize the expression in the parentheses w.r.t. $a\in{\bf A}$ , keeping the component $\theta$ of $b=(\theta,a)$ fixed. If $\theta\le\frac{x}{D}$, then we have
$$e^{-\alpha\theta}K+e^{-\alpha\theta}\min_{a\in{\bf A}} W^*_0(\min\{x-D\theta+a,L\})=e^{-\alpha\theta}K+e^{-\alpha\theta} \frac{D}{\alpha}e^{-\frac{\alpha L}{D}},$$
and, if $\theta\ge\frac{x}{D}$, then we have
\begin{eqnarray*}
	&&\frac{D}{\alpha}\left(e^{-\frac{\alpha x}{D}}-e^{-\alpha\theta}\right)+e^{-\alpha\theta}K+e^{-\alpha\theta}\min_{a\in{\bf A}}W^*_0(a)\\
	&=&\frac{D}{\alpha}\left(e^{-\frac{\alpha x}{D}}-e^{-\alpha\theta}\right)+e^{-\alpha\theta}K+e^{-\alpha\theta}\frac{D}{\alpha}e^{-\frac{\alpha L}{D}}
\end{eqnarray*}

In the first case, the infimum with respect to $\theta$ is achieved at  $\theta=\frac{x}{D}$ and equals $e^{-\frac{\alpha x}{D}}\left(K+\frac{D}{\alpha}e^{-\frac{\alpha L}{D}}\right)$, and in the second case, that infimum is achieved at $\theta=\infty$ because $K-\frac{D}{\alpha}+\frac{D}{\alpha} e^{-\frac{\alpha L}{D}}>0$, leading to $\frac{D}{\alpha}e^{-\frac{\alpha x}{D}}$.  Since $K+\frac{D}{\alpha} e^{-\frac{\alpha L}{D}}>\frac{D}{\alpha}$, eventually we see that the right-hand part of (\ref{e29}) equals $\frac{D}{\alpha}e^{-\frac{\alpha x}{D}}=W^*_0(x)$.\hfill$\square$\bigskip
	
	\underline{Proof of Theorem \ref{t6}.} 
	According to Lemma \ref{l7}, it is sufficient to check that the presented functions $W^*_g(\cdot)$ satisfy the optimality equation (\ref{e29}).
	
	(a) Assume that inequality (\ref{eqn2}) is not strict.
	The function (\ref{eqn3}) decreases when $x<a_g$ and increases when $x>a_g$ because
	$$\frac{d~W^*_g(x)}{dx}=\frac{gH}{\alpha}\left(1-e^{\alpha\frac{a_g-x}{D}}\right).$$
	Thus,
	$$\min_{x\in{\bf X}} W^*_g(x)=W^*_g(a_g)=\frac{a_ggH}{\alpha}.$$
	
	If $\theta\ge \frac{x}{D}$, then the expression in the parentheses of (\ref{e29}) for $W^*_g(\cdot)$ equals
	\begin{equation}\label{eqn8}
		\frac{D}{\alpha}(e^{-\alpha\frac{x}{D}}-e^{-\alpha\theta})+e^{-\alpha\theta}K+\frac{gHx}{\alpha}+\frac{gHD}{\alpha^2}(e^{-\alpha\frac{x}{D}}-1)+e^{-\alpha\theta}W^*_g(a).
	\end{equation}
	We substitute $\min_{a\in{\bf A}} W^*_g(a)=\frac{a_ggH}{\alpha}$ and obtain
	$$\frac{e^{-\alpha\theta}}{\alpha}[a_ggH+\alpha K-D]+\frac{D}{\alpha} e^{-\alpha\frac{x}{D}} +\frac{xgH}{\alpha}+\frac{DgH}{\alpha^2}(e^{-\alpha\frac{x}{D}}-1).$$
	Since $Hga_g+\alpha K-D\le 0$, the minimum with respect to $\theta\ge\frac{x}{D}$ is provided by $\theta=\frac{x}{D}$ leading to
	$$e^{-\alpha\frac{x}{D}}K+\frac{xgH}{\alpha} +\frac{DgH}{\alpha^2}(e^{-\alpha\frac{x}{D}}-1)+\frac{a_ggH}{\alpha} e^{-\alpha\frac{x}{D}}.$$
	Note, this optimal value of $\theta$ is unique if  inequality (\ref{eqn2}) is  strict; otherwise, the value of $\theta\ge\frac{x}{D}$ can be taken arbitrarily.
	By (\ref{eqn1}), the last term equals
	$$\left[\frac{DgH}{\alpha^2} e^{\frac{\alpha a_g}{D}} -K-\frac{DgH}{\alpha^2}\right] e^{-\alpha\frac{x}{D}}$$
	resulting in 
	$$\frac{xgH}{\alpha}-\frac{DgH}{\alpha^2}+\frac{DgH}{\alpha^2}e^{\alpha\frac{a_g-x}{D}}=W^*_g(x).$$
	
	It remains to show that, when $\theta<\frac{x}{D}$, the expression in the parentheses of (\ref{e29}) for $W^*_g(\cdot)$ is strictly larger than $W^*_g(x)$. Now this expression equals
	\begin{equation}\label{eqn7}
		e^{-\alpha\theta}K+\frac{gHx}{\alpha}(1- e^{-\alpha\theta})+gHD\left[e^{-\alpha\theta}\left(\frac{\theta}{\alpha}+\frac{1}{\alpha^2}\right)-\frac{1}{\alpha^2}\right]+e^{-\alpha\theta} W^*_g(\min\{x-D\theta+a,L\}).
	\end{equation}
	According to the above established properties of the function (\ref{eqn3}), the minimum with respect to $a\in{\bf A}$ is provided  by $a=a_g-x+D\theta$ if $x-D\theta\le a_g$ and by $a=0$ in case $x-D\theta\ge a_g$. 
	
	In the first case $\min_{a\in{\bf A}}W^*_g(\min\{x-D\theta+a,L\})=W^*_g(a_g)=\frac{a_ggH}{\alpha}$, and the derivative with respect to $\theta$ of  function (\ref{eqn7}) equals
	$= e^{-\alpha\theta}[-\alpha K+gHx-gHD\theta-a_ggH]<0$
	because $\theta D+a_g\ge x$. Thus, expression (\ref{eqn7}) is strictly larger than its value at $\theta=\frac{x}{D}$, which equals
	$$e^{-\alpha\frac{x}{D}} K+\frac{gHx}{\alpha}(1-e^{-\alpha\frac{x}{D}})+gHD\left(\frac{x}{D\alpha} e^{-\alpha\frac{x}{D}}+\frac{1}{\alpha^2} e^{-\alpha\frac{x}{D}}-\frac{1}{\alpha^2}\right)+e^{-\alpha\frac{x}{D}}\frac{a_ggH}{\alpha}.$$
	After we again use equation (\ref{eqn1}) for representing $\frac{a_ggH}{\alpha}$, the last expression becomes
	\begin{eqnarray*}
		&& \frac{xgH}{\alpha}-\frac{DgH}{\alpha^2}+e^{-\alpha\frac{x}{D}}\left[K+\frac{DgH}{\alpha^2}+\frac{DgH}{\alpha^2} e^{\frac{\alpha a_g}{D}}-K-\frac{DgH}{\alpha^2}\right]=W^*_g(x).
	\end{eqnarray*}
	
	In the second case, when $x-D\theta\ge a_g$, $\min_{a\in{\bf A}}W^*_g(\min\{x-D\theta+a,L\})=W^*_g(x-D\theta)$, and expression (\ref{eqn7}) equals
	\begin{eqnarray*}
		&&e^{-\alpha\theta}K+\frac{gHx}{\alpha}(1- e^{-\alpha\theta})+gHD\left(\frac{\theta}{\alpha}e^{-\alpha\theta}+\frac{1}{\alpha^2}e^{-\alpha\theta}
		-\frac{1}{\alpha^2}\right)\\
		&&+e^{-\alpha\theta}\left(\frac{(x-D\theta)gH}{\alpha}-\frac{DgH}{\alpha^2}+\frac{DgH}{\alpha^2} e^{-\alpha\frac{x-D\theta}{D}}\cdot e^{\alpha\frac{a_g}{D}}\right)\\
		&=& e^{-\alpha\theta}K+\frac{xgH}{\alpha}-\frac{DgH}{\alpha^2}+\frac{DgH}{\alpha^2} e^{\alpha\frac{a_g-x}{D}}>W^*_g(x).
	\end{eqnarray*}
	
	(b) Assume that inequality (\ref{eqn4}) is not strict, consider the case of $\theta\ge\frac{x}{D}$ and show that the presented functions $W^*_g(\cdot)$ satisfy the optimality equation (\ref{e29}).
	
	The expression in the parentheses of (\ref{e29}) is the same as in Item (a): see (\ref{eqn8}). But $\min_{a\in{\bf A}} W^*_g(a)$ is different because we have to use expression (\ref{eqn5}):
	$$\frac{d~W^*_g(a)}{da}=\frac{gH}{\alpha}-e^{-\alpha\frac{a}{D}}\left(\frac{gH}{\alpha}+1\right).$$
	Therefore, $\min_{a\ge 0} W^*_g(a)$ is provided by
	$$\hat a_g=\frac{D}{\alpha}\ln\left(1+\frac{\alpha}{gH}\right)$$
	and, keeping in mind that $a$ should not exceed $L$,
	$$\inf_{a\in{\bf A}}W^*_g(a)\ge W^*_g(\hat a_g)=\frac{\hat a_g gH}{\alpha}-\frac{DgH}{\alpha^2}+\left(\frac{DgH}{\alpha^2}+\frac{D}{\alpha}\right)\frac{\frac{gH}{\alpha}}{\frac{gH}{\alpha}+1}=\frac{\hat a_g gH}{\alpha}.$$
	Therefore, the infimum w.r.t. $a\in{\bf A}$ of expression (\ref{eqn8}) is not smaller than
	\begin{equation}\label{eqn10}
		\frac{e^{-\alpha\theta}}{\alpha}[-D+\alpha K+Hg\hat a_g]+\frac{xgH}{\alpha}-\frac{DgH}{\alpha^2}+\left(\frac{DgH}{\alpha^2}+\frac{D}{\alpha}\right) e^{-\alpha\frac{x}{D}}.
	\end{equation}
	
	Let us show that $[-D+\alpha K+Hg\hat a_g]\ge 0$. As the result, we shall deduce that the infimum w.r.t. $b=(\theta,a)\in{\bf B}$ of expression (\ref{eqn8}) equals $W^*_g(x)$ and is provided by $\theta=\infty$ and arbitrary $\hat a\in{\bf A}$.
	
	Firstly, note that the product $g\hat a_g$ increases with $g$ because the derivative
	$$\frac{d}{dg}\left(g\frac{D}{\alpha}\ln\left(1+\frac{\alpha}{Hg}\right)\right)=\frac{D}{\alpha}\left(\ln\left(1+\frac{\alpha}{Hg}\right)-\frac{\alpha}{Hg+\alpha}\right)$$
	is positive: its limit as $g\to\infty$ is zero and its derivative
	\begin{eqnarray*}
		\frac{D}{\alpha}\cdot\frac{d}{dg}\left(\ln\left(1+\frac{\alpha}{Hg}\right)-\frac{\alpha}{Hg+\alpha}\right)&=&\frac{D}{\alpha}\left(-\frac{Hg}{Hg+\alpha}\cdot\frac{\alpha}{Hg^2}+\frac{\alpha H}{(Hg+\alpha)^2}\right)\\
		&=&\frac{-D\alpha}{(Hg+\alpha)^2g}<0
	\end{eqnarray*}
	is negative. 
	
	According to Lemma \ref{l9}(c), since $\alpha K<D$, there exists the unique $\hat g$ for which $\alpha K+H\hat ga_{\hat g}=D$, and in the considered case $g\ge \hat g$. For $\hat g$, according to (\ref{eqn1}),  we have
	$$1+\frac{\hat g H}{\alpha}=\frac{\hat g H}{\alpha} e^{\frac{\alpha a_{\hat g}}{D}}~\Longrightarrow ~
	a_{\hat g}=\frac{D}{\alpha}\ln\left(1+\frac{\alpha}{\hat gH}\right)=\hat a_{\hat g},$$
	and so $[-D+\alpha K+H\hat g\hat a_{\hat g}]=0$. As was shown, the product $g\hat a_g$ increases with $g$. Thus, $[-D+\alpha K+Hg\hat a_g]\ge 0$ for $g\ge \hat g$, equivalently, when the non-strict inequality (\ref{eqn4}) holds. Thus, the minimum in (\ref{eqn10}) w.r.t.  $\theta\ge\frac{x}{D}$ is provided by $\theta=\infty$ and equals $W^*_g(x)$: see formula (\ref{eqn5}). Note, the optimal value $\theta=\infty$ is unique if inequality (\ref{eqn4}) is strict. Also note that, in case $gHa_g+\alpha K-D=0$, the value of $\theta\ge\frac{x}{D}$ can be taken arbitrarily: the corresponding value of $g$ was denoted above as $\hat g$ and $\hat a_{\hat g}=a_{\hat g}$, so that $[-D+\alpha K+Hg\hat a_{g}]=0$.
	
	It remains to show that, when $\theta<\frac{x}{D}$, the expression in the parentheses of (\ref{e29}) for $W^*_g(\cdot)$ is strictly larger than $W^*_g(x)$. It has the same form as in Item (a): see (\ref{eqn7}).
	According to the above established properties of  function (\ref{eqn5}), the minimum w.r.t. $a\in{\bf A}$ is provided  by $a=\hat a_g-x+D\theta$ if $x-D\theta\le \hat a_g$ and by $a=0$ in case $x-D\theta\ge \hat a_g$. 
	
	In the first case $\min_{a\in{\bf A}}W^*_g(\min\{x-D\theta+a,L\})=W^*_g(\hat a_g)=\frac{\hat a_ggH}{\alpha}$, and the derivative w.r.t. $\theta$ of  function (\ref{eqn7}) equals
	$ e^{-\alpha\theta}[-\alpha K+gHx-gHD\theta-\hat a_ggH]<0$
	because $D\theta +\hat a_g\ge x$. Thus, expression (\ref{eqn7}) is strictly larger than its value at $\theta=\frac{x}{D}$, which equals
	\begin{eqnarray*}
		&&e^{-\alpha\frac{x}{D}} K+\frac{gHx}{\alpha}(1-e^{-\alpha\frac{x}{D}})+gHD\left(\frac{x}{D\alpha} e^{-\alpha\frac{x}{D}}+\frac{1}{\alpha^2} e^{-\alpha\frac{x}{D}}-\frac{1}{\alpha^2}\right)+e^{-\alpha\frac{x}{D}}\frac{\hat a_ggH}{\alpha}\\
		&=&W^*_g(x)+e^{-\frac{\alpha x}{D}}\left(K-\frac{D}{\alpha}+\frac{\hat a_g gH}{\alpha}\right)\ge W^*_g(x).
	\end{eqnarray*}
	The inequality holds
	because, as was shown above,
	$K-\frac{D}{\alpha}+\frac{\hat a_g gH}{\alpha}\ge 0$ when (the non-strict) inequality (\ref{eqn4}) holds.
	
	In the second case, when $x-D\theta\ge \hat a_g$, $\min_{a\in{\bf A}}W^*_g(\min\{x-D\theta+a,L\})=W^*_g(x-D\theta)$, and expression (\ref{eqn7}) equals
	\begin{eqnarray*}
		&&e^{-\alpha\theta}K+\frac{gHx}{\alpha}(1- e^{-\alpha\theta})+gHD\left(\frac{\theta}{\alpha}e^{-\alpha\theta}+\frac{1}{\alpha^2}e^{-\alpha\theta}
		-\frac{1}{\alpha^2}\right)\\
		&&+e^{-\alpha\theta}\left(\frac{(x-D\theta)gH}{\alpha}-\frac{DgH}{\alpha^2}+\left(\frac{DgH}{\alpha^2}+\frac{D}{\alpha}\right) e^{-\alpha\frac{x-D\theta}{D}}\right)\\
		&=& e^{-\alpha\theta}K+\frac{xgH}{\alpha}-\frac{DgH}{\alpha^2}+\left(\frac{DgH}{\alpha^2}+\frac{D}{\alpha}\right) e^{-\frac{\alpha x}{D}}>W^*_g(x).
	\end{eqnarray*}
	
	(c) From equation (\ref{eqn1}) we have
	$$\frac{DgH}{\alpha^2} e^{\frac{\alpha a_g}{D}}=K+\frac{DgH}{\alpha^2}+\frac{Hga_g}{\alpha}=\frac{DgH}{\alpha^2}+\frac{D}{\alpha}.$$
	Hence functions (\ref{eqn3}) and (\ref{eqn5}) coincide.
	
	The uniform optimality of the deterministic stationary strategies of the form $f(x)=\left(\theta\ge \frac{x}{D},a_g\right)$, $x\in {\bf X}$ was established during the proof of items (a) and (b).
	\hfill$\square$\bigskip
	
	\underline{Proof of Theorem \ref{t8}.} 
	(a) According to Theorem \ref{t5}(b), $h(g)=W^*_g(0)-gd$.
	When $g>g_c$, i.e., by  Lemma \ref{l9}(c),
	$\alpha K+Hga_g>D$, we have by Theorem \ref{t6}(b) that
	$$\frac{dh}{dg}=\frac{d}{dg}\left[-\frac{DgH}{\alpha^2}+\frac{DgH}{\alpha^2}+\frac{D}{\alpha}-gd\right]=-d<0.$$
	
	Suppose $g_{min}\le g < g_c$,  i.e., $\alpha K+Hga_g<D$. According to Theorem \ref{t6}(a),
	\begin{eqnarray}
		\frac{dh}{dg} &=& \frac{d}{dg}\left[\frac{DgH}{\alpha^2}\left(e^{\alpha\frac{a_g}{D}}-1\right)-dg\right]
		= \frac{DH}{\alpha^2}\left(e^{\alpha\frac{a_g}{D}}-1\right)\nonumber\\
		&&-\frac{gH}{\alpha}e^{\alpha\frac{a_g}{D}}\left(\frac{D}{\alpha g}-\frac{a_g}{g(e^{\alpha\frac{a_g}{D}}-1)}\right)-d
		=\frac{Ha_g}{\alpha}\cdot\frac{e^{\alpha\frac{a_g}{D}}}{e^{\alpha\frac{a_g}{D}}-1} -\frac{DH}{\alpha^2}-d.\label{e16}
	\end{eqnarray}
	The second equality is by (\ref{e43}). 
	The function $u(z):=\frac{z\ln z}{z-1}$  increases  at $z>1$ because $\frac{du}{dz}=\frac{z-1-\ln z}{(z-1)^2}>0$. Therefore, since $a_g$ decreases with $g$ by Lemma \ref{l9}(a), at $g<g_c$ we have
	\begin{eqnarray*}
		\frac{Ha_g}{\alpha}\cdot\frac{e^{\alpha\frac{a_g}{D}}}{e^{\alpha\frac{a_g}{D}}-1}&=& \frac{HD}{\alpha^2} u(e^{\alpha\frac{a_g}{D}})>\frac{HD}{\alpha^2} u(e^{\alpha\frac{a_{g_c}}{D}})=\frac{Ha_{g_c}}{\alpha}\cdot\frac{e^{\alpha\frac{a_{g_c}}{D}}}{e^{\alpha\frac{a_{g_c}}{D}}-1}=d_c+\frac{DH}{\alpha^2}
	\end{eqnarray*}
	and, since $d\le d_c$, $\frac{dh}{dg}>0$.
	
	To summarise, the concave function $h(\cdot)$ attains its maximum at $g^*=g_c$, i.e., when $\alpha K+Hg_ca_{g_c}=D$. According to Theorem \ref{t6}(c), all the deterministic stationary strategies of the form $f(x)=(\theta(x)\ge \frac{x}{D},a_{g_c})$ are uniformly optimal in the sense of (\ref{e34}) at $g=g_c$.  In particular, so is $f^*$, and we are going  to show that ${\cal V}_1(0,f^*)=d$. To do so, note that
	$$\bar C_1(X_0,(\theta^*(X_0),a_{g_c}))=\bar C_1(0,(\theta^*(0),a_{g_c}))=0$$
	and $X_1=a_{g_c}$ with probability $e^{-\alpha\tau^*}$. (The complementary probability $(1-e^{-\alpha\tau^*})$ is for the event $X_1=\Delta$.) Thus,
	$$\bar C_1(X_1,(\theta^*(X_1),a_{g_c}))=e^{-\alpha\tau^*}\left[\frac{Ha_{g_c}}{\alpha}+\frac{HD}{\alpha^2}\left(e^{-\alpha\frac{a_{g_c}}{D}}-1\right)\right].$$
	And so on: for $i=2,3,\ldots$,
	$$\bar C_1(X_{i-1},(\theta^*(X_{i-1}),a_{g_c}))=\left[e^{-\alpha\left(\frac{a_{g_c}}{D}+\tau^*\right)}\right]^{i-2}\cdot
	e^{-\alpha\tau^*}\left[\frac{Ha_{g_c}}{\alpha}+\frac{HD}{\alpha^2}\left(e^{-\alpha\frac{a_{g_c}}{D}}-1\right)\right].$$
	leading to
	\begin{eqnarray*}
		{\cal V}_1(0,f^*) &=& \sum_{i=2}^\infty \left[e^{-\alpha\left(\frac{a_{g_c}}{D}+\tau^*\right)}\right]^{i-2}\cdot
		e^{-\alpha\tau^*}\left[\frac{Ha_{g_c}}{\alpha}+\frac{HD}{\alpha^2}\left(e^{-\alpha\frac{a_{g_c}}{D}}-1\right)\right]\\
		&=& \frac{e^{-\alpha\tau^*}\left[\frac{Ha_{g_c}}{\alpha}+\frac{HD}{\alpha^2}\left(e^{-\alpha\frac{a_{g_c}}{D}}-1\right)\right]}{1-e^{-\alpha\left(\frac{a_{g_c}}{D}+\tau^*\right)}}
		= \frac{\frac{Ha_{g_c}}{\alpha}+\frac{HD}{\alpha^2}\left(e^{-\alpha\frac{a_{g_c}}{D}}-1\right)}{e^{\alpha\tau^*}-e^{-\alpha\frac{a_{g_c}}{D}}}=d.
	\end{eqnarray*}
	
	From Theorem \ref{t6}(c) we have
	$${\cal V}_0(0,f^*)=W^*_{g^*}(0)-g^*d=\frac{D}{\alpha}-g^*d.$$
	Lemma \ref{l4} implies that $\mu^{f^*}({\bf X}\times{\bf B})<\infty$, and hence $\mu^{f^*}\in{\cal D}$.
	According to Remark \ref{rem3}(b),
	\begin{eqnarray*}
		\inf_{\mu\in{\cal D}} L_1(\mu,g^*)&=& W^*_{g^*}(0)-g^*d={\cal V}_0(0,f^*)\\
		&=& \int_{{\bf X}\times{\bf B}}\bar C_0(x,b)\mu^{f^*}(dx\times db)+g^*\left(\int_{{\bf X}\times{\bf B}}\bar C_1(x,b)\mu^{f^*}(dx\times db)-d\right):
	\end{eqnarray*}
	the last equality holds because $\int_{{\bf X}\times{\bf B}}\bar C_1(x,b)\mu^{f^*}(dx\times db)={\cal V}_1(0,f^*)=d$. Thus,
	$$\inf_{\mu\in{\cal D}} L_1(\mu,g^*)=L_1(\mu^{f^*},g^*)=\int_{{\bf X}\times{\bf B}}\bar C_0(x,b)\mu^{f^*}(dx\times db)$$
	and $\mu^{f^*}$ is an optimal (and feasible) solution to program (\ref{enum4p}) (equivalently, to program (\ref{e13pp}-\ref{e14pp})) by Theorem \ref{prop2}(b-ii). Finally, the deterministic stationary strategy $f^*$ is clearly induced by $\mu^{f^*}$, and thus is optimal (and feasible) in problem (\ref{PZZeqn02}) by Theorem \ref{t1}(c).
	
	When $d=d_c$, $\tau^*=0$, and this is the limiting case of Item (b): $\lim_{g\uparrow g_c}\frac{dh}{dg}=d_c-d=0$.
	
	(b) Since $\lim_{g\uparrow g_c}\frac{dh}{dg}=d_c-d$,  in case $d_c<d$ the dual functional $h(\cdot)$, which is concave and hence continuous at $g>0$, decreases at $g=g_c$. Recall that $\frac{dh}{dg}=-d<0$ at $g>g_c$. Therefore, the analytical maximum of the concave function $h(\cdot)$ is attained at $g^*=\hat g<g_c$ when $\frac{dh}{dg}=0$, i.e., when $a_{\hat g}=a^*$ is the solution of  equation (\ref{e15}) (see (\ref{e16})). Note that equation (\ref{e15}) has a unique positive solution because, as was shown above, function $u(z):=\frac{z\ln z}{z-1}$  increases  at $z>1$ and $\lim_{z\downarrow 1}u(z)=1$, $\lim_{z\to\infty} u(z)=\infty$ leading to $\frac{Ha}{\alpha}\cdot\frac{e^{\alpha\frac{a}{D}}}{e^{\alpha\frac{a}{D}}-1}\in\left(\frac{DH}{\alpha^2},\infty\right)$ when $a>0$. Positive $g$ and $a_g$ are in 1-1 correspondence by Lemma \ref{l9}(b). 
	
	Since $g_{min} < \hat g=g^*<g_c$, according to Lemma \ref{l9}(c),  $\alpha K+Hg^* a_{g^*}<D$. Therefore, by Theorem \ref{t6}(a), the deterministic stationary strategy $f^*$ is uniformly optimal in the sense of (\ref{e34}) at $g=g^*$. Calculation of ${\cal V}_1(0,f^*)$ is similar to that given above in part (a). One has to substitute zero for $\tau^*$ and replace $a_{g_c}$ with $a^*$:
	$${\cal V}_1(0,f^*)= \frac{\frac{Ha^*}{\alpha}+\frac{HD}{\alpha^2}\left(e^{-\alpha\frac{a^*}{D}}-1\right)}{1-e^{-\alpha\frac{a^*}{D}}}=d.$$
	
	From Theorem \ref{t6}(a) we have
	$${\cal V}_0(0,f^*)=W^*_{g^*}(0)-g^*d=\frac{Dg^*H}{\alpha^2}\left(e^{\alpha\frac{a^*}{D}}-1\right)-g^*d.$$
	
	To complete the prove, one has to repeat the reasoning at the end of part (a).
	\hfill$\square$

	\section{Statements and  Declarations}
	
	The author declare that no funds, grants, or other support were received during the preparation of this manuscript.
	
	The author has no relevant financial or non-financial interests to disclose.

\end{document}